\documentclass[a4paper,10pt,reqno,oneside]{amsart}
\usepackage{amsmath,amsthm,amssymb,enumerate,mathtools,stmaryrd,xcolor}
\usepackage{bbm}
\usepackage{libertine}

\usepackage{euscript,mathrsfs}
\usepackage[left=2.5cm,right=2.5cm,top=3cm,bottom=2.95cm]{geometry}
\usepackage[colorlinks=true, linktocpage=true, linkcolor=red!70!black, citecolor=green!50!black, urlcolor=black]{hyperref}
\usepackage[noabbrev]{cleveref}
\allowdisplaybreaks

\usepackage{enumitem}
\setenumerate{label={\rm (\alph{*})}}

\numberwithin{equation}{section}

\newcommand{\R}{\mathbb R}
\newcommand{\C}{\mathbb C}
\newcommand{\N}{\mathbb N}
\newcommand{\ball}{\mathrm{B}}
\newcommand{\A}{\mathscr{A}}
\newcommand{\Aop}{\mathbb{A}}

\newcommand{\Leb}{\mathscr{L}}
\newcommand{\Hd}{\mathscr{H}}

\newcommand{\lebe}{\operatorname{L}}
\newcommand{\sobo}{\operatorname{W}}
\newcommand{\cont}{\operatorname{C}}
\newcommand{\hold}{\operatorname{C}}
\newcommand{\BV}{\mathrm{BV}}
\newcommand{\loc}{\mathrm{loc}}
\newcommand{\D}{\mathrm{D}}
\DeclareMathOperator{\tr}{tr}

\DeclareMathOperator{\Lin}{Lin}
\newcommand{\dev}{\operatorname{dev}}
\newcommand{\sym}{\operatorname{sym}}
\newcommand{\intd}{\,\mathrm{d}}
\newcommand{\dx}[1]{\,\mathrm{d}#1}
\newcommand{\eps}{\varepsilon}
\newcommand{\e}{\mathbf{e}}
\newcommand{\del}{\partial}
\newcommand{\norm}[1]{\left\lVert#1\right\rVert}
\newcommand{\skalarProd}[2]{\big\langle#1,#2\big\rangle}
\newcommand{\abs}[1]{\left\lvert#1\right\rvert}
\newcommand{\F}{\mathscr{F}}
\newcommand{\bbone}{\text{\usefont{U}{bbold}{m}{n}1}}
\newcommand{\weakstar}{\overset{\ast}{\rightharpoonup}}
\MakeRobust{\bbone}

\newcommand{\Proj}{\mathbb{P}}
\newcommand{\diam}{\mathrm{diam}}
\newcommand{\rfrak}{\mathfrak{r}}
\renewcommand{\L}{\mathrm{L}}

\def\Xint#1{\mathchoice
	{\XXint\displaystyle\textstyle{#1}}%
	{\XXint\textstyle\scriptstyle{#1}}%
	{\XXint\scriptstyle\scriptscriptstyle{#1}}%
	{\XXint\scriptscriptstyle\scriptscriptstyle{#1}}%
	\!\int}
\def\XXint#1#2#3{{\setbox0=\hbox{$#1{#2#3}{\int}$ }
		\vcenter{\hbox{$#2#3$ }}\kern-.6\wd0}}

\def\dashint{\Xint-}

\colorlet{RED}{red}
\title{On $\BV^\Aop$-minimisers in two dimensions}
\newtheorem{thm}{Theorem}[section]
\newtheorem{cor}[thm]{Corollary}
\newtheorem{lem}[thm]{Lemma}
\newtheorem{prop}[thm]{Proposition}
\newtheorem{defi}[thm]{Definition}
\newtheorem{rem}[thm]{Remark}

\author[F. Eitler]{Ferdinand Eitler}
\address[Ferdinand Eitler]{University of Augsburg, Universitätsstr. 12, 86159 Augsburg, Germany}
\email{ferdinand.eitler@math.uni-augsburg.de}
\author[P. Lewintan]{Peter Lewintan}
\address[Peter Lewintan]{Institute for Analysis, Karlsruhe Institute of Technology, Englerstr. 2, 76131 Karlsruhe, Germany}
\email{peter.lewintan@kit.edu}

\begin{document}
\begin{abstract}
We investigate into the regularity of $\BV^\Aop$-minimisers for $\C$-elliptic differential operators $\Aop$ in $2$ dimensions. Our studies strongly rely on the special structure of such differential operators. The gradient integrability is established for the sharp ellipticity range known from the (symmetric) gradient case.

\end{abstract}

\maketitle
\section{Introduction}

\subsection{Overview of the problem.} Let $\Omega \subset \R^n$, $n\geq 2$ be an open and bounded Lipschitz domain with boundary $\del\Omega$. The goal of this paper is to give a contribution to the regularity theory for relaxed minimisers of 
\begin{align*}
	\F[u;\Omega] = \int_\Omega f(\Aop u)\dx{x}, \quad u\colon \R^n \to V,
\end{align*}
 where $\Aop$ is a $\ell$-th order, linear, constant-coefficient, homogeneous differential operator on $\R^n$ between two real finite dimensional vector spaces $V$ and $W$, i.e.
\begin{align}\label{eq:DifferentialOperator}
\Aop u\coloneqq \sum_{\abs{\alpha}=\ell}\Aop_{\alpha}\partial^\alpha u, \quad u\colon\R^n\to V,
\end{align}
with linear maps $\Aop_{\alpha}:V\to W$ and multi-indices $\alpha\in \N_0^n$. Moreover,  $f\colon W\to\R$ is an integrand which will be specified later. In case of the full or symmetric gradient, i.e.~$\Aop u =\D u$ or $\Aop u=\sym(\D u)\coloneqq\frac{1}{2}(\D u+\D u^\top)$ respectively, such functionals appear in the continuum description for modelling elastic and plastic effects for solids and fluids, where $u$ is then the corresponding displacement or velocity field, cf.~\cite{FuchsSeregin}. Usually, plasticity is described by means of linear growth densities $f\colon W\to \R$,  meaning that there exist two constants $c_1,c_2>0$ such that
\begin{align}\label{eq:LinGrowth}
c_1\abs{P}\le f(P) \le c_2 (1+\abs{P})\quad\mbox{for all}\quad P\in W.
\end{align}
This growth condition guarantees that $\F[-;\Omega]$ is well-defined on the space 
\begin{align*}
	\sobo^{\Aop,1}(\Omega) &\coloneqq \{v\in \lebe^1(\Omega;V): \Aop v\in \lebe^1(\Omega;W)\},
\end{align*}
being the natural generalisation of $\sobo^{1,1}(\Omega;\R^n)$. We can equip $\sobo^{\Aop,1}(\Omega)$ with the canonical norm $\|u\|_{\sobo^{\Aop,1}(\Omega)}\coloneqq \|u\|_{\lebe^1(\Omega;V)} + \|\Aop u\|_{\lebe^1(\Omega;W)}$, turning it into a Banach space.  Of course there holds $\sobo^{1,1}(\Omega;\R^N)= \sobo^{\D,1}(\Omega)$ with $V=\R^N$ and $W=\R^{N\times n}$ in the full gradient case. As the space $\sobo^{\Aop,1}(\Omega)$ is not reflexive, we cannot expect to have existence of minimisers using the direct method of the Calculus of Variations. Thus, a relaxation to the space $\BV^\Aop(\Omega)$ of functions of bounded $\Aop$-variation is necessary to ensure existence of minimisers. The latter function space is defined by
\begin{equation*}
	\BV^\Aop (\Omega) \coloneqq \{v\in\lebe^1(\Omega; V): \Aop v \in \mathrm{RM}_{\rm fin}(\Omega;W)\},
\end{equation*}
where $\mathrm{RM}_{\rm fin}(\Omega;W)$ denotes the space of $W$-valued finite Radon measures, cf.~\cite{BDG}. 
 Another drawback compared to the superlinear growth case $1<p<\infty$ is, that there exists \textbf{no} constant $c=c(\Aop)>0$ such that  
\begin{align*}
	\|\D u\|_{\lebe^1(\Omega;\Lin(\R^n;V))}\leq c\|\Aop u\|_{\lebe^1(\Omega;W)}
\end{align*}
holds for all $u\in\hold^\infty_c(\Omega;V)$ which is known under the name \emph{Ornstein-type non-inequality} in the literature, cf.~\cite[Thm. 1.3]{KirchheimKristensen}. Therefore, we have the strict inclusions $\sobo^{1,1}(\Omega;V) \subsetneq \sobo^{\Aop,1}(\Omega)$ and $\BV(\Omega;V)\subsetneq \BV^\Aop(\Omega)$, where $\BV(\Omega;V) = \BV^\D(\Omega)$ is the usual space of functions of bounded variation. Since the full distributional derivatives of $\BV^\Aop$-functions do not belong to $\mathrm{RM}_{\rm fin}(\Omega;\Lin(\R^n;V))$, we cannot apply known results for linear growth functionals involving the full gradient. Hence, we are interested in the question under which conditions on the integrand $f$ we have $\BV_\loc(\Omega;V)$ or even $\sobo^{1,1}_\loc(\Omega;V)$ regularity for relaxed minimisers. In particular, we are seeking for a parallel regularity theory to what is known in the full or symmetric gradient case for linear growth functionals. To state our main result we next introduce the precise framework. 
\subsection{Existence results for relaxed minimisers} Throughout this paper we assume that the Dirichlet datum satisfies $u_0\in\sobo^{\Aop,1}(\Omega)$. As usual, we define $\sobo^{\Aop,1}_0(\Omega)$ to be the closure of the test-functions $\hold^\infty_c(\Omega;V)$ with respect to the $\|\cdot\|_{\sobo^{\Aop,1}(\Omega)}$-norm. Apparently, a crucial assumption for the differential operator $\Aop$ of the form \eqref{eq:DifferentialOperator} is the notion of $\C$-ellipticity. Namely, if we associate the corresponding Fourier symbol map
\begin{equation*}
	\Aop[\xi]=\sum_{\abs{\alpha}=\ell}\xi^\alpha\Aop_\alpha: V\to W, \quad \xi\in\R^n
\end{equation*}
to $\Aop$, we call the operator $\Aop$ \emph{elliptic} if
\begin{equation*}
	\forall \xi\in\R^n\backslash\{0\}:\qquad \Aop[\xi]:V\to W \text{ is injective},
\end{equation*}
and $\C$-\emph{elliptic} if
\begin{equation*}
	\forall \xi\in\C^n\backslash\{0\}:\qquad \Aop[\xi]:V+\mathrm{i}V\to W+\mathrm{i}W \text{ is injective}.
\end{equation*}
It turns out that $\C$-ellipticity is equivalent to the fact of having a trace operator, cf.~\cite{BDG} as well as (compact) embeddings, cf.~\cite{GmRa1}. Moreover, since $\C$-ellipticity is equivalent to having a finite dimensional null-space, cf.~\cite[Thm. 2.6]{BDG} it follows from \cite{GmRa2,Raita} that $u\in\BV^\Aop$ is also $\lebe^p$-differentiable, which will be important to have the decomposition \eqref{eq:LebesgueRadonNikodym_Au} below. The following characterisation of $\C$-elliptic differential operators is essentially due to \textsc{Smith} \cite{Smith} but we also refer to \textsc{Ka\l{}amajska} \cite{Kalamajska} and \textsc{Gmeineder} et al. \cite[Prop. 3.2]{GRV}:
\begin{lem}\label{lem:Cellipt}
	$\Aop$ is $\C$-elliptic if and only if there exists another linear, constant-coefficient, homogeneous differential operator $\mathbb{L}$ on $\R^n$ and a number $d\in\N$ such that $\D^d=\mathbb{L}\circ\Aop$.
\end{lem}

In the following, we focus on the case $\ell=1$ and $n=2$ and consider first-order $\C$-elliptic differential operators $\Aop$ on $\R^2$ that are induced by an (orthogonal) projection $\A\colon\R^{2\times 2}\to\R^{2\times2}$ (like $\A\in\{\sym,\dev\}$ denoting the symmetric and deviatoric part map), i.e. the induced differential operator on $\R^2$ from $\R^2$ to $\R^{2\times2}$ is given via $\Aop\coloneqq \A[\cdot\otimes \nabla]$. In this setting, these operators possess a certain structure and we have by \cite[Prop 4.1]{GLN1} the following decomposition
\begin{align}\label{eq:complementarypart}
	P = \mathfrak{L}(\A[P]) + \gamma(P) \mathfrak{G}\quad\mbox{for all}\quad P\in\R^{2\times 2}, 
\end{align}
where $\mathfrak{L}\colon\R^{2\times 2}\to \R^{2\times 2}$ and $\gamma\colon\R^{2\times2}\to \R$ are linear maps, and $\mathfrak{G}\in\mathrm{GL}(2)$. Furthermore, the integer $d$ from Lemma \ref{lem:Cellipt} can be specified:
\begin{lem}\label{lem:Cellipt2d}
	Let $\Aop\coloneqq\A[\cdot\otimes\nabla]$ be a first-order $\C$-elliptic differential operator on $\R^2$ induced by a projection $\A:\R^{2\times2}\to\R^{2\times2}$. Then there exists a first-order linear constant-coefficient homogeneous differential operator $\mathbb{L}$ such that $\D^2=\mathbb{L}\circ\Aop$.
\end{lem}
The proof of this lemma will be given below in Subsection \ref{sec:Korn_in_Orlicz}. From now on, we consider the variational principle
\begin{align*}
	\mbox{to minimise}\quad \mathscr{F}[u;\Omega ]\coloneqq \int_\Omega f(\Aop u) \dx{x} \quad\mbox{over} \quad \mathscr{D}_{u_0} \coloneqq u_0 + \sobo^{\Aop,1}_0(\Omega),
\end{align*}
where $f\colon W\to \R$ is a convex integrand satisfying the linear growth condition \eqref{eq:LinGrowth} and $\Aop=\A[\cdot\otimes\nabla]$ is a $\C$-elliptic differential operator induced by a projection $\A$ as above. Recalling the discussion in  \cite[Sec. 5]{BDG}, we notice that it is not too restrictive to assume $W=\R^{2\times 2}_\A$, where 
\begin{align*}
	\R^{2\times 2}_{\A}\coloneqq\{P\in\R^{2\times 2}:\A[P]=P\}=\mathrm{span}\{v\otimes_{\Aop} \xi\coloneqq \A[v\otimes \xi], v,\xi\in\R^2\}. 
\end{align*}
To overcome the missing weak compactness in $\sobo^{\Aop,1}(\Omega)$, we relax the functional $\F$ to the space $\BV^\Aop(\Omega)$ by setting
\begin{align}\label{eq:Relaxation}
	\overline{\F}_{u_0}[v;\Omega] = \int_\Omega f(\A[\D v])\dx{x} 
	&+ \int_\Omega f^\infty\left(\frac{\dx{\Aop^sv}}{\dx{\abs{\Aop^sv}}}\right)\dx{\abs{\Aop^sv}}\notag \\
	&+\int_{\del\Omega} f^\infty (\tr(u_0-v)\otimes_{\Aop}\nu_{\del\Omega})\dx{\Hd^{n-1}},
\end{align}
where $\nu_{\partial\Omega}$ denotes the unit outward normal field along $\partial\Omega$ and the behaviour of the integrand $f$ at infinity is encoded in the \emph{recession function} defined by
\begin{align}\label{eq:RecessionFunction}
	f^\infty(z) \coloneqq \lim_{s\to\infty} \frac{f(sz)}{s}\quad\mbox{for}\quad z\in \R^{2\times2}_\A.
\end{align}
Moreover, we have used the Lebesgue-Radon-Nikod\'ym decomposition which allows to split $\Aop u$ into 
\begin{equation}\label{eq:LebesgueRadonNikodym_Au}
	\Aop u = \Aop^a u + \Aop^s u = \frac{\dx{\Aop^a u}}{\dx{\Leb^n}}\Leb^n + \frac{\dx{\Aop^s u}}{\dx{\abs{\Aop^s u}}}\abs{\Aop^s u} = \A[\D u] \Leb^n + \frac{\dx{\Aop^s u}}{\dx{\abs{\Aop^s u}}}\abs{\Aop^s u}, 
\end{equation}
 where the density of $\Aop^s u$ with respect to $\Leb^n$ can be identified with $\A[\D u]$, which is the part map of the approximate gradient $\D u$. 
For the precise definition of all building blocks of the above formula we refer to \Cref{sec:Preliminaries}. Next we introduce the notion of minimality, which we will use in the sequel:
\begin{defi}[$\BV^\Aop$- and local $\BV^{\Aop}$-minimisers]\label{defi:minimiser} ~
	\begin{enumerate}
		\item Given $u_0\in \BV^{\Aop}(\Omega)$, a map $u\in \BV^\Aop(\Omega)$ is called \emph{$\BV^\Aop$-minimiser with respect to the Dirichlet datum $u_0$} if we have 
		$\overline{\F}_{u_0}[u;\Omega] \leq \overline{\F}_{u_0}[w;\Omega]$ for all $w\in \BV^{\Aop}(\Omega)$. 
		\item A map $u\in \BV^\Aop_\loc(\Omega)$ is called \emph{local $\BV^\Aop$-minimiser} if we have $\overline{\F}_u[u; U] \leq \overline{\F}_u[w;U]$ for all $U\Subset \Omega$ with Lipschitz boundary $\del U$ and all $w\in \BV^{\Aop}(U)$.
	\end{enumerate}
\end{defi}
We note, that it follows from the definition, that any $\BV^\Aop$-minimiser is also a local $\BV^\Aop$-minimiser. Under our assumptions, it is shown in \cite[Thm. 5.3]{Gmeineder} or \cite[Thm. 3.3]{Piotr} that there exists a $\BV^\Aop$-minimiser $u$, which is the weak${}^\ast$-limit of a minimising sequence in $\mathscr{D}_{u_0}$ without having a relaxation gap, i.e.
\begin{equation}\label{eq:ConsistencyRelation}
\inf_{\mathscr{D}_{u_0}} \F[-; \Omega] = \min_{\BV^{\Aop}(\Omega)}\overline{\F}_{u_0}[-;\Omega] = \overline{\F}_{u_0}[u;\Omega]. 
\end{equation}
Moreover, since the recession function is positively $1$-homogeneous and hence, not strictly convex, it turns out, that the whole functional $\overline{\F}_{u_0}[-;\Omega]$ is not strictly convex on $\BV^\Aop_{(\loc)}(\Omega)$. As a consequence, (local) $\BV^\Aop$-minimiser may be non-unique in general. 

In order to conclude Sobolev regularity we have to impose a suitable ellipticity condition on $f$, namely we consider the notion of $\mu$-ellipticity described in the upcoming subsection.
\subsection{The scale of $\mu$-elliptic variational integrands and previous results.} In the context of convex linear growth integrands it is common to consider $\mu$-elliptic variational integrands. More precisely, given $1<\mu<\infty$,  we say that some $\hold^2$-integrand $f\colon\R^{2\times2}_\A\to \R$ is \emph{$\mu$-elliptic} if there are constants $0<\lambda \leq \Lambda<\infty$ such that
\begin{align}\label{eq:MuEllipticity}
	\lambda \,\frac{\abs{\xi}^2}{(1+\abs{P}^2)^\frac{\mu}{2}}\leq \langle \D^2f(P)\xi,\xi\rangle\leq\Lambda\,\frac{\abs{\xi}^2}{(1+\abs{P}^2)} \quad\mbox{for all}\quad P,\xi\in\R^{2\times 2}_{\A}.
\end{align}
We briefly notice, that the condition $\mu>1$ is necessary, since otherwise the integrand is no longer of linear growth, cf.~\cite[Rem. 4.14]{BEG} for more details. In case of the full gradient, i.e.~the framework of $\BV$-spaces, \textsc{Bildhauer} proved in all dimensions that at least one $\BV$-minimiser is of class $\sobo^{1,p}_\loc(\Omega;\R^N)\cap \sobo^{1,1}(\Omega;\R^N)$ for $p>1$ if $1<\mu +\frac{2}{n}$ . The $\sobo^{1,1}$-regularity remains true for $1+\frac{2}{n}\leq \mu\leq 3$ under the additional assumption of local boundedness, which can be justified by means of maximum principles or Moser-type iterations. We refer to \cite{Bildhauer} for the precise statements and more details. In \cite{BeckSchmidt}, \textsc{Beck} and \textsc{Schmidt} extended the above Sobolev regularity result to \emph{all} $\BV$-minimisers using the Ekeland variational principle, recalling that relaxed minimisers may be non-unique. We briefly indicate that so far there is no result for the autonomous Dirichlet problem available dealing with the case $\mu>3$. However, we mention \cite{Neumann} for a result concerning the Neumann problem as well as \cite{BildhauerLecNotes} for a counter-example for non-autonomous variational integrals in case $\mu>3$. Let us also mention \cite{Bildhauer2d} dealing with two-dimensional radially symmetric variational integrals as well as \cite{Lukas_Paul_etc} for a result concerning non-autonomous integrands in the full gradient case. Finally, we note that $\mu$-elliptic integrands show a similar behaviour to one of $(p,q)$-growth and we refer to the discussion in \cite{BEG} and \cite{DarkGuide} for a detailed overview including many additional references.

The first contribution in direction of the symmetric gradient can be found in \cite{GK}, proving higher integrability and Sobolev regularity for a non-optimal range $1<\mu<1+\frac{1}{n}$ by means of fractional methods. The same result was later extended by \textsc{Gmeineder} in \cite{Gmeineder} to the range $1<\mu<1+\frac{2}{n}$ known from the full gradient case, using techniques which heavily exploit the specific structure of the symmetric gradient. Finally, assuming additional local boundedness for the minimiser, it is shown in \cite{BEG} that $\sobo^{1,1}$-regularity holds for the full range $1<\mu\leq 3$, hence giving a picture to that in the $\BV$-case.  Again, the specific structure of the symmetric gradient plays a key role in the proof.

 So far the only result treating more general $\C$-elliptic operators was given by \textsc{Wozniak} in \cite{Piotr}, generalising the higher integrability and Sobolev regularity for the non-optimal range $1<\mu<1+\frac{1}{n}$, exploiting that the fractional techniques are more flexible. At the present stage it seems hopeless to transfer the methods from \cite{Gmeineder} and \cite{BEG} to the general $\C$-elliptic case, since the structure of the differential operator plays a major role. However, in two dimensions all $\C$-elliptic operators induced by a projection $\A$ share the main structural condition \eqref{eq:complementarypart}, which allows to generalise the main result of \cite{Gmeineder} for the range $1<\mu<1+\frac{2}{n}$. This will be the main theorem of our paper and reads as follows:
 \begin{thm}[Gradient integrability]\label{thm:MainTheorem}
 	Let $u_0\in \sobo^{\Aop,1}(\Omega)$ and $f\in\hold^2(\R^{2\times 2}_\A)$ be  a variational integrand satisfying \eqref{eq:LinGrowth} and \eqref{eq:MuEllipticity} for $\mu\in(1,2)$. Then every local $\BV^\Aop$-minimiser $u$ of $\F$ is of class $\sobo^{\Aop,1}(\Omega;\R^2)\cap \sobo^{1,q}_\loc(\Omega;\R^2)$ for all $q\in[1,\infty)$. More precisely, for every subset $U\Subset \Omega$ there exists a constant $c=c(\Lambda,\mu,\Aop)>0$ such that whenever $\ball_{2r}(x_0)\Subset U$ we have 
 	\begin{align}\label{eq:FinalEstimateMainTheorem}
 		\begin{split}
 		&\|\D u\|_{\exp\lebe^{\frac{2-\mu}{3-\mu}}(\ball_{r}(x_0); \R^{2\times 2})} \\
 		&\hspace{1cm}\leq c \left(\bigg(1+\frac{1}{r^2}\bigg)\bigg[\bigg(1+\frac{1}{r^2}\bigg) \abs{\Aop u}(\overline{\ball}_{2r}(x_0)) + \dashint_{\ball_{2r}(x_0)} \frac{\abs{u}}{r} \dx{x}\bigg]^{\frac{1}{2-\mu}} + \dashint_{\ball_{r}(x_0)} \frac{\abs{u}}{r}\dx{x}\right).
 		\end{split}
 	\end{align}
 \end{thm}
 \begin{rem}
 	The arguments we use in our proof also show that this statement holds true in \emph{all} dimensions for the part map $\A=\dev$, cf. Remarks \ref{rem:dev1} and \ref{rem:dev2}.
 \end{rem}
 At this stage, we briefly comment on the proof of our main theorem \eqref{thm:MainTheorem}. In order to overcome the possible non-uniqueness we will adapt the vanishing viscosity method in conjunction with the Ekeland variational principle. We start with an arbitrary $\BV^\Aop$-minimising, whose approximation is a minimising sequence for the original functional. Since this sequence does not satisfy the corresponding Euler-Lagrange inequality, we add an additional regularising term to our functional and apply the Ekeland variational principle in order to derive a minimising sequence of almost-minimisers $(v_j)_{j\in\N}$ close to the original sequence. The latter comes naturally with an Euler-Lagrange inequality, which turns out to be sufficient to derive weighted  second-order estimates, which will be the key ingredient for showing gradient integrability. In order to derive these bounds, we apply the Ekeland variational principle in the space $\sobo^{-2,1}$, which may be surprising at first glimpse: Let us recall that for the symmetric gradient there holds
 \begin{equation*}
 	\partial_{i}\partial_j u_k=\partial_i \sym\D u^{(jk)}-\partial_k \sym\D u^{(ij)}+\partial_j \sym\D u^{(ik)}.
 \end{equation*}
for all $u\in\hold^\infty_c(\R^n;\R^n)$ and $i,j,k\in\{1,\dots,n\}$.  By  \Cref{lem:Cellipt2d} such a representation remains valid in our situation. In order to obtain the weighted estimates, we have to test the Euler-Lagrange inequality in a differentiated version by second-order quantities behaving approximately like $\Delta v_j$. At this stage the usage of $\sobo^{-1,1}$ as proposed in \cite{BeckSchmidt} would destroy any estimates, since we only have 
\begin{equation*}
	\|\Delta v_j\|_{\sobo^{-1,1}}\lesssim \|\D v_j\|_{\L^1}.
\end{equation*}
The right-hand side of the latter cannot be controlled because of Ornstein's Non-Inequality. However, as in the case of the symmetric gradient, cf.~\cite{Gmeineder, BEG}, the space $\sobo^{-2,1}$ turns out to be sufficient, because in this case we have 
\begin{equation*}
	\|\Delta v_j\|_{\sobo^{-2,1}}\lesssim \|v_j\|_{\lebe^1}.
\end{equation*}
Moreover, the explicit structure of the operator is \textit{essential}, which already appeared in the framework of the symmetric gradient. Let us emphasise that the decomposition \eqref{eq:complementarypart} is the main building block in order to conclude the weighted second-order estimates. Without the latter, it currently appears hopeless to apply the described method to obtain a similar result that includes $\C$-elliptic differential operators in all dimensions. 

\section{Preliminaries}\label{sec:Preliminaries}
We mainly use standard notation.  We equip $\R^n$ and $\R^{n\times m}$ with the usual Euclidean or Frobenius norm, respectively and denote both by $\abs{\,\cdot\,}$. The latter is induced by the matrix inner product $\langle A, B\rangle \coloneqq A:B\coloneqq \tr(A^\top B)$ for $A,B\in\R^{n\times m}$. 
Moreover, we write $\ball_r(x_0)=\{x\in\R^n\colon \abs{x-x_0}< r\}$ for an open ball with radius $r>0$ centred at $x_0\in\R^n$. 
$\Leb^n$ and $\Hd^{n-1}$ indicate the $n$-dimensional Lebesgue and $(n-1)$-dimensional Hausdorff measure, respectively and we simply write $\dx{x}$, when integrating with respect to $\Leb^n$. Moreover we set $\omega_n \coloneqq \Leb^n(\ball_1(0))$. For a (Lebesgue) measurable set $U\subset\R^n$ with $0<\Leb^n(U)<\infty$ and $f\in\lebe^1_\loc(\R^n;\R^m)$ we denote the average by
\begin{align*}
	\dashint_U f(x)\dx{x} \coloneqq \frac{1}{\Leb^n(U)}\int_\Omega f(x)\dx{x}.
\end{align*}
Furthermore, $V$ and $W$ are two finite dimensional normed spaces and with some abuse of notation we will denote both norms by $\abs{\,\cdot\,}$, as it is clear from the context. We use the notation $a\otimes b= ab^\top$ for the usual tensor product of two vectors $a\in\R^n$ and $b\in\R^m$. Finally, $c>0$ denotes a generic constant, whose value may change from line to line and whose dependencies are usually indicated, whereas the precise value is irrelevant. Finally, we state two elementary estimates for linear growth integrands, which we will use in the proof of the main theorem.
\begin{lem}[{\cite[Lem. 5.2]{Giusti}, \cite[Lem. 2.8]{BeckSchmidt}}]
	Suppose that $f\colon \R^{2\times 2}_\A\to \R$ is a convex $\hold^1$-function which satisfies the linear growth condition \eqref{eq:LinGrowth}. Then we have the following statements:
	\begin{enumerate}
		\item For all $z\in\R^{2\times 2}_\A$ there holds $\abs{\D f(z)} \leq c_2$ and in particular $\mathrm{Lip}(f)\leq c_2$.
		\item For all $z\in\R^{2\times 2}_\A$ there holds $\langle \D f(z),z\rangle\geq c_1\abs{z}-c_2 $.
	\end{enumerate}
\end{lem}

\subsection{Function spaces and related topics} For an open set $\Omega\subset\R^n$, $1\leq p\leq \infty$ and $k\in\N$ we denote by $\lebe^p(\Omega)$ and $\sobo^{k,p}(\Omega)$ the Lebesgue and Sobolev spaces, cf.~e.g.~\cite{Giusti}. Without further comment we also use $V$-valued versions of these spaces. For a finite Radon measure $\mu\in\mathrm{RM}_{\rm fin}(\Omega;V)$ we denote by $\abs{\mu}$ its total variation. Moreover, we have the Lebesgue-Radon-Nikod\'ym decomposition 
\begin{equation}\label{eq:LebesgueRadonNikodym}
	\mu = \mu^a + \mu^s = \frac{\dx{\mu^a}}{\dx{\Leb^n}}\Leb^n + \frac{\dx{\mu^s}}{\dx{\abs{\mu^s}}}\abs{\mu^s},
\end{equation}
into absolutely continuous and singular part with respect to $\Leb^n$. 
\subsubsection{Functions of bounded $\Aop$-variation}
As already mentioned above, a function of bounded $\Aop$-variation, is a function $u\in\lebe^1(\Omega;V)$ such that $\Aop u$ is a finite Radon-measure. We notice, that a linear, homogeneous, first-order differential operator $\Aop$ can be also characterised by its part map $\A\in\Lin(V\otimes \R^n; W)$ via $\Aop u = \A[\D u ]$. 
In order to introduce a norm on $\BV^\Aop(\Omega)$, we define the \emph{dual} or \emph{formally adjoint operator} $\Aop^\ast$ as the differential operator $\Aop^\ast$ on $\R^n$ from $W$ to $V$ given by
\begin{equation*}
	\Aop^\ast v\coloneqq\sum_{\alpha=1}^n \Aop_\alpha^\ast \del_\alpha v,\quad v\colon \R^n\to W,
\end{equation*}
where $\Aop_\alpha^\ast$ denotes the adjoint map of $\Aop_\alpha$. This allows us to introduce the \emph{total \(\Aop\)-variation} of \(u \in \lebe^1_{\mathrm{loc}}(\Omega; V)\) by
\begin{equation*}
	|\Aop u|(\Omega) \coloneqq \sup \left\{ \int_{\Omega} \langle u, \Aop^\ast \varphi \rangle \dx{x}\colon \varphi \in \hold^1_c(\Omega; W), \, \|\varphi\|_{\infty} \leq 1 \right\},
\end{equation*}
which allows us to characterise the space \(\BV^\Aop(\Omega)\) equivalently as the set of functions \(u\) having finite \(\Aop\)-variation.
 The local variant $\BV^\Aop_\loc(\Omega)$ is defined in the obvious manner. Furthermore, the space $\BV^\Aop$ is a Banach space with respect to the norm
\begin{equation*}
	\|u\|_{\BV^\Aop(\Omega)}\coloneqq \|u\|_{\lebe^1(\Omega;V)} + \abs{\Aop u}(\Omega).
\end{equation*}
Similar to the $\BV$-setting there are three different notions of convergence available on $\BV^\Aop$. Namely, given $v\in \BV^\Aop(\Omega)$ and a sequence $(v_j)_{j\in\N}$ in  $\BV^\Aop(\Omega)$ we say that $v_j$ converges to $v$ as $j\to\infty$
	\begin{itemize}
	\item in the \emph{weak$^\ast$-sense}, in symbols $v_j\weakstar v$, if $v_j\to v$ strongly in $\lebe^1(\Omega;V)$ and $\Aop v_j \weakstar \Aop v$ in the sense of $W$-valued Radon measures.
	\item in the \emph{$\Aop$-strict sense}, if $v_j\to v$ strongly in $\lebe^1(\Omega;V)$ and $\abs{\Aop v_j}(\Omega)\to \abs{\Aop v}(\Omega)$ as $j\to \infty$.
	\item in the \emph{$\Aop$-area-strict sense} if $v_j\to v$ strongly in $\lebe^1(\Omega;V)$ and $\langle\Aop v\rangle(\Omega)\to \langle \Aop v\rangle(\Omega)$ as $j\to \infty$, where we have abbreviated
	\begin{equation*}
		\langle \Aop v\rangle (\Omega)\coloneqq  \int_\Omega \sqrt{1+\abs{\Aop v}^2}\dx{x} +\abs{\Aop^s v}(\Omega)\quad\mbox{for}\quad u\in \BV^\Aop(\Omega).
	\end{equation*} 
	\end{itemize}
	In the last formula we have used the Lebesgue-Radon-Nikod\'ym decomposition \eqref{eq:LebesgueRadonNikodym_Au}.
\subsubsection{Orlicz-Sobolev-spaces} In order to derive estimates for the full gradient from the one for $\Aop u$ as needed in \eqref{eq:FinalEstimateMainTheorem}, we use a suitable Korn-type inequality, cf.~\Cref{sec:Korn_in_Orlicz}. Towards this aim, let $\Phi\colon [0,\infty)\to [0,\infty)$ be a Young-function, meaning
\begin{equation*}
\Phi(t) = \int_0^t \phi(s)\dx{s} \quad\mbox{for}\quad t\geq 0,
\end{equation*}
where $\phi:[0,\infty)\to [0,\infty]$ is a non-decreasing, left-continuous function, which is neither identical to $0$ nor $\infty$. We then define the \emph{Lebesgue-Orlicz space} as the space
\begin{equation*}
	\lebe^\Phi(\Omega;V)\coloneqq \{u\colon \Omega\to \R^n\mbox{  measurable  }\colon \|u\|_{\lebe^\Phi(\Omega;V)}<\infty\}, 
\end{equation*}
where $\|\cdot\|_{\lebe^\Phi(\Omega;V)}$ denotes the \emph{Luxemburg norm} given by
\begin{equation*}
	\|u\|_{\lebe^\Phi(\Omega;V)}\coloneqq \inf\left\{\lambda>0\colon \int_\Omega \Phi\left(\frac{\abs{u}}{\lambda}\right)\dx{x}\leq 1\right\}.
\end{equation*}
We abbreviate $\exp\lebe^\beta(\Omega;V)\coloneqq \lebe^{\Phi_\beta}(\Omega;V)$ for the Orlicz-space corresponding to $\Phi_\beta(t)\coloneqq \exp(t^\beta)$ with $t\geq 0$ and $\beta>0$. Moreover, we define the \emph{Orlicz-Sobolev space} $\sobo^{1,\Phi}(\Omega;V)$ as
\begin{equation*}
	\sobo^{1,\Phi}(\Omega;V)\coloneqq \{u\in \lebe^\Phi(\Omega;V)\colon u\mbox{ is weakly differentiable and }\D u\in \lebe^\Phi(\Omega;\Lin(\R^n;V))\},
\end{equation*}
which is a Banach space when equipped with the norm 
\begin{equation*}
	\|u\|_{\sobo^{1,\Phi}(\Omega;V)} \coloneqq \|u\|_{\lebe^\Phi(\Omega;V)} + \|\D u\|_{\lebe^\Phi(\Omega;\Lin(\R^n, V))}.
\end{equation*}

\subsubsection{Negative Sobolev spaces} For $k\in\N$ we define the \emph{negative Sobolev space} $\sobo^{-k,1}(\Omega;V)$ as the set of all $V$-valued distributions $T\in\mathscr{D}'(\Omega;V)$ which admit the representation 
\begin{equation*}
	T= \sum_{\abs{\alpha}\leq k} \del^\alpha T_\alpha \quad\mbox{with}\quad T_\alpha\in \lebe^{1}(\Omega;V),\,\alpha\in\N^n_0 \quad\mbox{and}\quad \abs{\alpha}\leq k.
\end{equation*}
This space becomes a Banach space when endowed with the norm 
\begin{equation*}
	\|T\|_{\sobo^{-k,1}(\Omega;V)}\coloneqq \inf \bigg\{ \sum_{\abs{\alpha}\leq k} \|T_\alpha\|_{\lebe^1(\Omega;V)}: T=\sum_{\abs{\alpha}\leq k} \del^\alpha T_\alpha, \hspace{0.1cm} T_\alpha\in \lebe^1(\Omega;V)\bigg\}.
\end{equation*}
As a consequence of \cite[Lem. 2.5]{Gmeineder} (cf.~ also \cite[Sec. 3.2.3]{BEG}), we recall for $s\in\{1,\dots,n\}$ and $h>0$ the estimates
\begin{align}
	&\|\del_s w \|_{\sobo^{-2,1}(\Omega;\R^2)} \leq \|w\|_{\sobo^{-1,1}(\Omega;\R^2)}\leq \|w\|_{\lebe^1(\Omega;\R^2)},\\
	& \|\del_s w\|_{\sobo^{-1,1}(\Omega;\R^2)}\leq \|w\|_{\lebe^1(\Omega;\R^2)}\\
	&\|\triangle_{s,h} w\|_{\sobo^{-1,1}(\Omega_h;\R^2)}\leq\|w\|_{\lebe^1(\Omega;\R^2)}\label{eq:EstimateNegSob3}
\end{align}
where $\Omega_h \coloneqq \{x\in\Omega: \mathrm{dist}(x,\del\Omega)>h\}$ and $\triangle_{s,h} w$ denotes the finite difference quotient, i.e.~
\[
\triangle_{s,h} w \coloneqq \frac{w(x+he_s)- w(x)}{h}, \quad x\in\Omega_h. 
\]

\subsection{(Lower semi-)continuity results}
Since the distributions $\Aop u$ are merely Radon measures we need lower semi-continuity results for functionals defined on measures. In this way, we define based on the Lebesgue-Radon-Nikod\'ym decomposition \eqref{eq:LebesgueRadonNikodym}, another measure $f(\mu)$ through
\begin{equation}\label{eq:FunctionalDefinedOnMeasures}
	f(\mu)(U) \coloneqq \int_U f\left(\frac{\dx{\mu^a}}{\dx{\Leb^n}}\right)\dx{x}+ \int_U f^\infty\left(\frac{\dx{\mu^s}}{\dx{\abs{\mu^s}}}\right)\dx{\abs{\mu^s}}\quad\mbox{for Borel subsets  }U\subset\Omega.
\end{equation}
We notice, that if $f$ is of linear growth in the sense of \eqref{eq:LinGrowth}, then the recession function $f^\infty$ of $f$ defined through \eqref{eq:RecessionFunction} is a well-defined, 1-homogeneous, lower semi-continuous and convex function taking finite values. The last condition can no longer be guaranteed, if one drops linear growth assumption from above.  The main (lower semi-)continuity result for functionals of the form \eqref{eq:FunctionalDefinedOnMeasures} is due to \textsc{Reshetnyak}, which we state in the following tailor-made version:
\begin{thm}[Reshetnyak (lower semi-)continuity theorem, {\cite{Reshetnyak},\cite[Thm. 2.4]{BeckSchmidt}}]\label{thm:Reshetnyak}
	Let \,$\Omega\subset\R^n$ be an open and bounded set and let $f\colon \R^{2\times 2}_\A\to \R$ be a convex function of linear growth from below. For functions $u,u_1,u_2,\dots\in \BV^{\Aop}(\Omega)$ we have
	\begin{enumerate}
		\item if $u_j\weakstar u$ in the weak$^\ast$-sense in $\BV^{\Aop}(\Omega)$ then 
		\begin{equation*}
			f(\Aop u)(\Omega) \leq \liminf_{j\to\infty} f(\Aop u_j)(\Omega),
		\end{equation*}
		\item and if $u_j\to u$ in the $\Aop$-area-strict sense in $\BV^{\Aop}(\Omega)$ then
		\begin{equation*}
			f(\Aop u)(\Omega) = \lim_{j\to\infty} f(\Aop u_j)(\Omega).
		\end{equation*}
	\end{enumerate}
\end{thm}
\noindent In particular, recalling \eqref{eq:LebesgueRadonNikodym_Au}, we infer that the relaxed functionals \eqref{eq:Relaxation} are lower semi-continuous with respect to the weak$^\ast$-convergence as well as continuous with respect to the $\Aop$-area-strict convergence, if $f\colon\R^{2\times2}_\A\to \R$ is a convex integrand of linear growth in the sense of \eqref{eq:LinGrowth}.

\pagebreak \noindent Finally, we state a customised version of \cite[Lem. 2.6]{Gmeineder}, which will be necessary when applying the Ekeland variational principle \Cref{prop:Ekeland} later on. 
\begin{lem}\label{lem:LowerSemicontinuity}
	Let $\Omega\subset\R^2$ be an open and bounded set with Lipschitz boundary, $1<q<\infty$ and $k\in\N$. Moreover, assume that $\mathfrak{f}\colon\R^{2\times 2}_\A\to \R_{\geq 0}$ is a convex function that satisfies $\tfrac{1}{c} \abs{P}^q \leq \mathfrak{f} (P)\leq c(1+\abs{P}^q)$ for some $c>0$ and all $P\in\R^{2\times 2}_\A$. Then for every $u_0\in\sobo^{1,q}(\Omega;\R^2)$, the functional 
	\begin{align*}
		\mathfrak{F}[u] \coloneqq 
		\begin{cases}
			\displaystyle \int_\Omega \mathfrak{f}(\Aop u) \dx{x}\quad &\mbox{if } u\in \mathscr{D}_{u_0}\coloneqq u_0 + \sobo^{\Aop,1}(\Omega;\R^2)\\
			+\infty\quad &\mbox{if } u\in\sobo^{-k,1}(\Omega;\R^2)\setminus \mathscr{D}_{u_0}
		\end{cases}
	\end{align*}
is lower semi-continuous with respect to the norm topology on $\sobo^{-k,1}(\Omega;\R^2)$. 
\end{lem}
\subsection{The Ekeland variational principle}

In order to overcome possible non-uniqueness phenomena for $\BV^\Aop$-minimisers, we use the Ekeland variational principle to construct a suitable minimising sequence. We will need it in the following version:

\begin{prop}[Ekeland {\cite{Ekeland}},{\cite[Thm. 5.6, Rem. 5.5]{Giusti}}] \label{prop:Ekeland}
	Let $(X,d)$ be a complete metric space and  $F:X\to \R\cup\{\infty\}$ be a lower semi-continuous function with respect to the metric topology, which is bounded from below and not identically $+\infty$. Suppose that for some $u\in X$ and some $\eps>0$, there holds $F[u]\leq \inf_X F +\eps$. Then there exists $v\in X$ such that 
	\begin{enumerate}
		\item $d(u,v) \leq \sqrt{\eps}$. 
		\item $F[v] \leq F[u]$ and 
		\item for all $w\in X$ there holds $F[v]\leq F[w]+\sqrt{\eps}\, d(v,w)$. 
	\end{enumerate}
\end{prop}

\section{Proof of the main theorem}
\subsection{Ekeland approximation procedure}
We observe that \Cref{thm:MainTheorem} contains a local statement and therefore, whenever  $u\in\BV^\Aop_\loc(\Omega)$ is a local $\BV^\Aop$-minimiser of $F$, the restriction $u\vert_U\in \BV^\Aop$ to a relatively compact set $U\Subset\Omega$ with Lipsichitz boundary $\del U$, is a $\BV^\Aop$-minimiser with respect to its own boundary values. Hence, it is not too restrictive to assume 
\begin{equation*}
	u=u_0\in \BV^\Aop(\Omega). 
\end{equation*}
 Although many steps of the approximation are standard, we provide full details for the sake of completeness. The construction is divided into several steps. \\

\emph{Step 1: Construction of a regular minimising sequence subject to smooth boundary values.} Using \cite[Proposition 4.24]{BDG} we get a smooth sequence $(u_j)_{j\in\N}$ in $u_0 + \hold^\infty_c(\Omega;\R^2)$  such that 
\begin{align}\label{eq:AreaStrictConvergence}
	u_j \to u\quad\mbox{in the $\Aop$-area-strict topology on $\BV^{\Aop}$ as } j\to\infty.
\end{align}
As a consequence of \Cref{thm:Reshetnyak} and the discussion below, the functional $\overline{\F}_{u_0}[-;\Omega]$ is continuous with respect to the $\Aop$-area-strict topology on $\BV^{\Aop}(\Omega)$. Therefore, using the consistency relation from \eqref{eq:ConsistencyRelation}, we obtain 
\begin{align*}
	\lim_{j\to\infty} \F[u_j; \Omega] = \lim_{j\to\infty} \overline{\F}_{u_0}[u_j;\Omega] = \overline{\F}_{u_0}[u; \Omega] = \min_{\BV^{\Aop}(\Omega)} \overline{\F}_{u_0}[-;\Omega] = \inf_{\mathscr{D}_{u_0}}\F[-;\Omega]. 
\end{align*}
Hence, $(u_j)_{j\in\N}$ is a minimising sequence for $\F[-;\Omega]$ in $\mathscr{D}_{u_0}$. Passing to a non-relabeled subsequence if required, we may assume 
\begin{align}\label{eq:minimisingSequence}
	\F[u_j;\Omega] \leq \inf_{\mathscr{D}_{u_0}} \F[-;\Omega] + \frac{1}{8j^2}\quad\mbox{for all}\quad j\in\N.
\end{align}
Since $\Aop$ is $\C$-elliptic we can apply \cite[Theorem 4.17]{BDG} to conclude the existence of a continuous trace-operator $\tr\colon\sobo^{\Aop,1}(\Omega;\R^2) \to \lebe^1(\del\Omega, \Hd^{n-1})$. Therefore, we can find a compactly-supported extension $\overline{u}_0\in\sobo^{\Aop,1}(\R^2;\R^2)$ of $u_0$, cf.~also \cite[Corollary 4.21]{BDG}.
Next, we choose $\eps_j>0$ in a way, such that the mollification of \,$\overline{u}_0$ with a scaled radial standard mollifier $\varrho_{\eps_j}$ provides us with a sequence $(u_j^{\del\Omega})_{j\in\N}$ in $\sobo^{1,2}(\Omega;\R^2)$ defined by $u_j^{\del\Omega}\coloneqq (\overline{u}_0\ast \varrho_{\eps_j})\vert_{\Omega}$ which satisfying
\begin{align}\label{eq:EstimateMollification}
	\|u_j^{\del\Omega} - u_0\|_{\sobo^{\Aop,1}(\Omega)} \leq \frac{1}{8 \mathrm{Lip}(f) j^2}. 
\end{align}
This allows us to define approximate Dirichlet classes through $\mathscr{D}_j\coloneqq u_j^{\del\Omega} + \sobo^{\Aop,1}_0(\Omega;\R^2)$. Setting $\tilde{u}_j \coloneqq u_j -u_0 + u_j^{\del\Omega}\in\mathscr{D}_j$, we conclude using \eqref{eq:EstimateMollification} that 
\begin{align}\label{eq:minimisingSequenceWithSmoothBoundary}
	\|\tilde{u}_j - u_j\|_{\sobo^{\Aop,1}(\Omega)} = \|u_j^{\del\Omega} - u_0 \|_{\sobo^{\Aop,1}(\Omega)}\leq \frac{1}{8\mathrm{Lip}(f) j^2}. 
\end{align} 
We now show, that the infimum of $\F[-;\Omega]$ over $\mathscr{D}_{u_0}$ can be approximated by the corresponding one over the Dirichlet class $\mathscr{D}_j$, which is almost attained for $\tilde{u}_j$. Towards this claim, we observe that 
\begin{equation*}
	\big \vert \F[v;\Omega] - \F[w;\Omega]\big\vert \leq \mathrm{Lip}(f) \|\A[\D v]- \A[\D w] \|_{\lebe^1(\Omega;\R^{2\times2}_\A)},\quad \forall  v,w \in \sobo^{\Aop,1}(\Omega),
\end{equation*}
and thus, we obtain 
\begin{align*}
	&\F[u_0+\varphi; \Omega] \leq \F[u_j^{\del\Omega}+\varphi; \Omega] + \mathrm{Lip}(f) \|\A[\D u_0] - \A [\D u_j^{\del\Omega}]\|_{\L^1(\Omega; \R^{2\times 2}_\A)}\\
	&\F[u_j^{\del\Omega}+\varphi; \Omega ]\leq \F[u_0+\varphi; \Omega] + \mathrm{Lip(f)} \|\A[\D u_0] - \A[\D u_j^{\del\Omega}]\|_{\L^1(\Omega;\R^{2\times 2}_\A)}
\end{align*}
for every $\varphi\in\sobo^{\Aop,1}_0(\Omega)$. Taking into account \eqref{eq:minimisingSequenceWithSmoothBoundary}, we can take the infimum on the right-hand side over all $\varphi\in\sobo^{\Aop,1}_0(\Omega)$ to deduce
\begin{align}\label{eq:InfimaApproximation}
	\bigg\vert \inf_{\mathscr{D}_{u_0}}\F[-;\Omega] - \inf_{\mathscr{D}_{j}} \F[-;\Omega]\bigg\vert\leq \frac{1}{8j^2}.
\end{align}
Finally, we notice that the choice $\varphi = u_j-u_0$ in combination with \eqref{eq:minimisingSequence} results in 
\begin{align}\label{eq:AlmostAttainment}
	\F[\tilde{u}_j;\Omega] \leq \F[u_j; \Omega] + \frac{1}{8j^2} \leq \inf_{\mathscr{D}_{u_0}}\F[-;\Omega]+ \frac{1}{4j^2},
\end{align}
and hence, the infimum of $\F[-;\Omega]$ over $\mathscr{D}_j$ is almost attained by $\tilde{u}_j$. This step will be crucial, when verifying the assumptions of the Ekeland variational principle from \Cref{prop:Ekeland} in the next step. \\

\emph{Step 2: Definition of regularised functionals.} In this part, we implement a vanishing viscosity approximation by adding a quadratic regularisation term. More precisely, for $P\in\R^{2\times 2}_{\A}$ we define
\begin{align}\label{eq:DefinitionOfIntegrands}
	f_j(P)\coloneqq f(P) + \frac{1}{2A_j j^2}(1+\abs{P}^2)\quad\mbox{with}\quad A_j \coloneqq 1+\int_\Omega (1+ \abs{\A[\D\tilde{u}_j]}^2)\dx{x},
\end{align}
leading to
\begin{align}\label{eq:sigma}
	\abs{\D f_j(P)} \leq \abs{\D f(P)} + \frac{1}{2A_j j^2}\abs{P}\overset{\eqref{eq:LinGrowth}}{\leq} c_2+ \frac{1}{2A_j j^2}\abs{P}.
\end{align}
We finally introduce the regularised functionals in terms of $\tilde{u}_j$ by
\begin{align*}
\mathscr F_j [v;\Omega] =  \begin{cases}\displaystyle
	 \mathscr F[v;\Omega] + \frac{1}{A_j j^2} \int_\Omega (1+\abs{\A[\D v] }^2)\dx{x}\quad&\mbox{if}  \quad v\in\mathscr{D}_j\\
	+\infty\quad&\mbox{if}\quad v\in\sobo^{-2,1}(\Omega;\R^2) \setminus \mathscr{D}_j. 
\end{cases}
\end{align*}

\emph{Step 3: The Ekeland approximation.} As already announced at the end of the first step,  \eqref{eq:AlmostAttainment} is crucial to verify the assumption of the Ekeland variational principle. Indeed, we obtain that $\tilde{u}_j$ is $j^{-1}$ close to the infimum of $\F_j[-;\Omega]$ in $\sobo^{-2,1}(\Omega;\R^2)$ because of
\begin{align} \label{eq:CloseToInf_Ekeland}
	\F_j[\tilde{u}_j;\Omega] \leq \F[\tilde{u}_j;\Omega] + \frac{1}{2j^2} 
	&\overset{\eqref{eq:AlmostAttainment}}\leq \inf_{\mathscr{D}_{u_0}} \F[-;\Omega] + \frac{3}{4j^2}  \\
	&\overset{\eqref{eq:InfimaApproximation}}{\leq} \inf_{\mathscr{D}_j} \F[-;\Omega] + \frac{1}{j^2} \leq \inf_{\sobo^{-2,1}(\Omega;\R^2)} \F_j[-;\Omega] + \frac{1}{j^2}\notag .
\end{align}
As $\sobo^{-2,1}(\Omega;\R^2)$ is a Banach space and $\F_j[-;\Omega]$ is clearly not identical $\infty$, as well as lower semi-continuous on $\sobo^{-2,1}(\Omega;\R^2)$ by \Cref{lem:LowerSemicontinuity} with $k=2$ and $\mathfrak{f} = f_j$, the Ekeland variational principle from \Cref{prop:Ekeland} is applicable and provides us with a sequence $(v_j)_{j\in\N}$ of almost-minimisers satisfying
\begin{align}
	\|v_j - \tilde{u}_j \|_{\sobo^{-2,1}(\Omega;\R^2)} &\leq \frac{1}{j} \label{eq:Ekeland_1}\\
	\F_j [v_j; \Omega] &\leq \F_j[w;\Omega] + \frac{1}{j} \|v_j -w\|_{\sobo^{-2,1}(\Omega;\R^2)}\quad\mbox{for all}\quad w\in \sobo^{-2,1}(\Omega;\R^2)\label{eq:Ekeland_2}. 
\end{align}

\pagebreak\noindent 
We proceed by proving useful properties of the Ekeland-type approxmiation sequence:
\begin{prop}[Properties of the Ekeland sequence]\label{prop:EkelandSequence}
	Let $(v_j)_{j\in\N}$ be the Ekeland-type approximiation sequence from above. Then for every $j\in\N$ we have the following estimates:
	\begin{align}
		\int_\Omega \abs{\A[\D v_j]} \dx{x}& \leq \frac{1}{c_1} \bigg(\inf_{\mathscr{D}_{u_0}} \F[-;\Omega] + \frac{2}{j^2}\bigg),\label{eq:EkelandSequence_Property_1}\\
		\frac{1}{2A_j j^2} \int_\Omega (1+\abs{\A[\D v_j]}^2)\dx{x}& \leq \frac{2}{j^2}\label{eq:EkelandSequence_Property_2}. 
	\end{align}
\end{prop}
\begin{proof}
	For arbitrary $j\in\N$ we test the almost-minimality condition \eqref{eq:Ekeland_2} with $w=\tilde{u}_j$, which results, after exploiting \eqref{eq:Ekeland_1} and \eqref{eq:CloseToInf_Ekeland}, in
	\begin{align}\label{eq:FinitenessOfFunctionalForEkelandSeq}
		\F_j[v_j;\Omega] \leq \F_j[\tilde{u}_j,\Omega] + \frac{1}{j^2}\leq \inf_{\mathscr{D}_{u_0} }\F[-;\Omega] + \frac{2}{j^2}.
	\end{align}
	In particular, $\F_j[{v_j};\Omega]$ is finite and therefore, \eqref{eq:EkelandSequence_Property_1} follows directly from the lower bound of the linear growth condition \eqref{eq:LinGrowth}. By definition of the functional $\F_j$ we obtain $v_j\in\mathscr{D}_j\cap \sobo^{1,2}(\Omega;\R^2)$ and, using \eqref{eq:FinitenessOfFunctionalForEkelandSeq}, also 
	\begin{align*}
		\inf_{\mathscr{D}_j} \F[-;\Omega] \leq \F[v_j; \Omega] \leq \F_j[v_j;\Omega] \leq 	\inf_{\mathscr{D}_{u_0} }\F[-;\Omega] +\frac{2}{j^2}.
	\end{align*}
	In view of \eqref{eq:InfimaApproximation} we conclude
	\begin{align*}
		\frac{1}{2A_j j^2} \int_\Omega (1+\abs{\A[\D v_j]}^2) \dx{x}  = \left(\F_j[v_j;\Omega] - \F[v_j; \Omega]\right) \leq \frac{9}{8j^2}\le\frac{2}{j^2},
	\end{align*}
	which is exactly the second estimate \eqref{eq:EkelandSequence_Property_2}. This finishes the proof. 
\end{proof}

\noindent Introducing the abbreviation $\sigma_j\coloneqq \D f_j(\A[\D v_j])$, we immediately observe $\A[\sigma_j]=  \sigma_j$, since $\A$ is a projection on $\R^{2\times 2}_\A$ and $\sigma_j\in\R^{2\times 2}_{\A}$. 
The key ingredient to derive regularity of all generalised minimisers will be a Euler-Lagrange inequality, following from the almost-minimality condition \eqref{eq:Ekeland_2}. More precisely, we have 
\begin{thm}[Euler-Lagrange inequality]\label{thm:EulerLagrangeInequality}
	Let $(v_j)_{j\in\N}$ be the Ekeland-type approximation sequence from \Cref{prop:EkelandSequence}. Then the following \emph{Euler-Lagrange inequality} holds true for all $j\in\N$:
	\begin{align}\label{eq:EulerLagrangeInequality}
		\left\vert \int_\Omega \skalarProd{\sigma_j}{\D \varphi}\dx{x}\right\vert =\left\vert \int_\Omega \skalarProd{\sigma_j}{\A[\D \varphi]}\dx{x}\right\vert \leq \frac{1}{j}\|\varphi\|_{\sobo^{-2,1}(\Omega;\R^2)} \qquad \forall\ \varphi\in\sobo^{1,2}_0(\Omega;\R^2).
	\end{align}
\end{thm}
\begin{proof}
	The first relation follows in view of
	\begin{align*}
		\skalarProd{\sigma_j}{\D \varphi}=\skalarProd{\A[\sigma_j]}{\D\varphi}=\skalarProd{\sigma_j}{\A[\D\varphi]}.
	\end{align*}
	The remaining part follows essentially the same steps as in \cite[(4.13)]{Gmeineder} or \cite[Lemma 4.4]{Piotr}. To this end, we consider an arbitrary test function $\varphi\in\sobo^{1,2}_0(\Omega;\R^2)$, which is admissible because of the quadratic growth, cf.~\eqref{eq:DefinitionOfIntegrands}.  For $\theta>0$ we test the almost-minimality condition \eqref{eq:Ekeland_2} with $w=v_j\pm \theta \varphi$. Dividing the resulting inequalities by $\theta$ yields
	\begin{align*}
		\frac{\F_j[v_j\pm \theta\varphi; \Omega]- \F_j [\varphi;\Omega]}{\theta}\leq \frac{1}{j}\|\varphi\|_{\sobo^{-2,1}(\Omega;\R^2)}
	\end{align*}
	and passing to the limit $\theta\searrow 0$, we obtain the two one-sided Gâteaux-differentials of the functional $\F_j$ in directions $\pm\varphi$, which can be computed explicitly. Thus, we can finish the proof:
	\begin{equation*}
		\lim_{\theta\searrow 0} \frac{\F_j[v_j\pm \theta\varphi; \Omega]- \F_j [\varphi;\Omega]}{\theta} = \left.\frac{\dx{}}{\dx{\theta}}\right\vert_{\theta=0} \F_j[v_j \pm \theta \varphi;\Omega] = \pm\int_\Omega  \skalarProd{\sigma_j}{\A[\D \varphi]}\dx{x}.\qedhere
	\end{equation*}
\end{proof}

\subsection{Non-uniform estimates}

In order to justify subsequent computations, we first derive an existence result for second derivatives of $v_j$, which is non-uniformly in $j\in\N$. This can be achieved in a fairly standard manner, by testing the Euler-Lagrange inequality \eqref{eq:EulerLagrangeInequality} with difference quotients. Nevertheless, we have to be more careful as usual, since we deal with differential inequalities instead of equations. 

\begin{lem}[Non-uniform second-order estimates]\label{lem:NonUniformSecEstimates}
	Let $f\in\cont^2(\R^{2\times2}_\A)$ satisfy \eqref{eq:LinGrowth} and for some $\Lambda\in (0,\infty)$ the estimate 
	\begin{align}\label{eq:MuElliptic_UpperBound}
		0\leq\skalarProd{\D^2 f(P)\xi}{\xi}\leq \Lambda \frac{\abs{\xi}^2}{(1+\abs{P}^2)^\frac{1}{2}}\quad\mbox{for all}\quad P,\xi\in\R^{2\times 2}_\A. 
	\end{align}
Moreover, let  $(v_j)_{j\in\N}$ be the Ekeland-type approximation sequence from \Cref{prop:EkelandSequence}.  Then we have $v_j\in\sobo^{2,2}_\loc(\Omega;\R^2)$. 
\end{lem}
\begin{proof}
	The proof is similar to \cite[Lemma 4.2]{Gmeineder}, but for the sake of completeness we give the whole argument. Let $x_0\in\Omega$ be a point and fix two radii $0<r<R<\mathrm{dist}(x_0,\del\Omega)$. Moreover, let $s\in\{1,2\}$,  $0<h<\frac{1}{2}(\mathrm{dist}(x_0,\del\Omega)-R)$ and a localisation function $\varrho\in\hold^\infty_c(\Omega;[0,1])$ with $\chi_{\ball_{r}(x_0)}\leq\varrho\leq \chi_{\ball_{R}(x_0)}$. We test the Euler-Lagrange inequality \eqref{eq:EulerLagrangeInequality} with $\varphi\coloneqq \triangle_{s,-h}(\varrho^2 \triangle_{s,h} v_j)\in\sobo^{1,2}_0(\Omega;\R^2)$ and after a discrete integration by parts we obtain
	\begin{align}\label{eq:EulerLagrange_TestedWithDiffQuot}
		\left\vert \int_\Omega \skalarProd{\triangle_{s,h}\D f_j(\A[\D v_j])}{\A[\D(\varrho^2 \triangle_{s,h}v_j)]}\dx{x} \right\vert \leq\frac{1}{j^2} \|\triangle_{s,-h}(\varrho^2 \triangle_{s,h}v_j)\|_{\sobo^{-2,1}(\Omega;\R^2)}. 
	\end{align}
	By the fundamental theorem of calculus in combination with the chain rule, we observe
	\begin{align*}
		\triangle_{s,h} \D f_j(\A[\D v_j])(x) 
		 = \int_0^1 \D^2 f_j\big(\A[\D v_j](x) + th \triangle_{s,h} \A[\D v_j](x)\big)\dx{t} \,\triangle_{s,h} (\A[\D v_j])(x),
	\end{align*}
	which motivates for $\Leb^n$-a.e. $x\in \ball_{r}(x_0)$ the definition of bilinear forms 
	$\mathscr{B}_{j,s,h}(x):\R^{2\times 2}_\A \times \R^{2\times 2}_\A\to \R$
	through
	 \begin{align*}
	 	\mathscr{B}_{j,s,h}(x)[\eta,\xi]\coloneqq \int_0^1 \skalarProd{\D^2 f\big(\A[\D v_j] (x) + th \triangle_{s,h}(\A[\D v_j])(x)\big)\eta }{\xi}\dx{t}\quad\mbox{for all}\quad \eta,\xi\in\R^{2\times 2}_\A. 
	 \end{align*}
 	We note that each $\mathscr{B}_{j,s,h}(x)$ is an elliptic bilinear form, since, by the very definition of $f_j$ from \eqref{eq:DefinitionOfIntegrands}, together with \eqref{eq:MuElliptic_UpperBound}, the following estimates hold:
 	\begin{align}\label{eq:BoundsBilinearform}
 		\frac{\abs{\xi}^2}{A_j j^2} \leq \mathscr{B}_{j,s,h}(x) [\xi,\xi] \leq \bigg(\Lambda + \frac{1}{A_j j^2}\bigg) \abs{\xi}^2. 
 	\end{align}
 	In particular, a suitable version of the Cauchy-Schwarz inequality is available. Using the product rule $\A[\D( \varrho^2 \triangle_{s,h}v_j)] = \varrho^2 \A[\D (\triangle_{s,h} v_j)] + \A[2\varrho\nabla\varrho\otimes \triangle_{s,h} v_j]$ and rearranging terms in \eqref{eq:EulerLagrange_TestedWithDiffQuot}, gives 
 	\begin{align*}
 		\mathrm{I} 
 		&\coloneqq \int_\Omega \skalarProd{\triangle_{s,h}\D f_j (\A[\D v_j])}{\varrho^2 \A[\D (\triangle_{s,h} v_j)]}\dx{x} \notag \\
 		& \leq - \int_\Omega  \skalarProd{\triangle_{s,h}\D f_j(\A[\D v_j])}{ \A[2\varrho\nabla\varrho\otimes \triangle_{s,h} v_j]}\dx{x}
 		+\frac{1}{j} \|\triangle_{s,-h}(\varrho^2 \triangle_{s,h} v_j)\|_{\sobo^{-2,1}(\Omega;\R^n)} \\
 		&\eqqcolon \mathrm{II} + \mathrm{III}\notag. 
 		\end{align*}
 	Rewriting term $\mathrm{I}$ and using the lower bound from \eqref{eq:BoundsBilinearform}, we infer
 	\begin{align}\label{eq:BoundBelow_I}
 		\mathrm{I} = \int_\Omega \mathscr{B}_{j,s,h}(x) \big[\varrho \A[\D (\triangle_{s,h} v_j)], \varrho \A[\D (\triangle_{s,h} v_j)]\big] \dx{x} \geq \frac{1}{A_j j^2} \int_\Omega \abs{\varrho \A[\D (\triangle_{s,h} v_j)]}^2 \dx{x}. 
 	\end{align}
 	Moreover, using the Cauchy-Schwarz inequality, we estimate
 	\begin{align*}
 		\mathrm{II} &= -\int_\Omega \mathscr{B}_{j,s,h}(x) \big[\varrho \A[\D (\triangle_{s,h} v_j)], \A[2\nabla \varrho\otimes \triangle_{s,h} v_j]\big]\dx{x} \\
 		&\leq \frac{1}{2} \int_\Omega \mathscr{B}_{j,s,h}(x) \big[\varrho \A[\D (\triangle_{s,h} v_j)], \varrho \A[\D(\triangle_{s,h} v_j)]\big]\dx{x} \\
	 	&\hspace{2cm}+\frac{1}{2} \int_\Omega \mathscr{B}_{j,s,h}(x) \big[ \A[2\nabla \varrho\otimes \triangle_{s,h}v_j], \A[2 \nabla \varrho \otimes \triangle_{s,h} v_j]\big]\dx{x}.  
 	\end{align*}
 	While the first term on the right hand side can be absorbed into $\mathrm{I}$, the second can be treated with the help of the upper bound of \eqref{eq:BoundsBilinearform}. More precisely, we have 
 	\begin{align*}
 		\frac{1}{2} &\int_\Omega \mathscr{B}_{j,s,h}(x) \big[ \A[2\nabla \varrho\otimes \triangle_{s,h}v_j], \A[2 \nabla \varrho \otimes \triangle_{s,h} v_j]\big]\dx{x} \\
 		&\leq 2\bigg(\Lambda + \frac{1}{A_j j^2}\bigg)\|\nabla\varrho\|_{\lebe^\infty(\Omega;\R^2)}^2 \int_{\ball_{R}(x_0)} \abs{\triangle_{s,h} v_j}^2 \dx{x} \\
 		&\leq 2\bigg(\Lambda + \frac{1}{A_j j^2}\bigg)\|\nabla\varrho\|_{\lebe^\infty(\Omega;\R^2)}^2~ \|v_j\|_{\sobo^{1,2}(\Omega;\R^2)}^2,
 	\end{align*}
 	which is finite as $v_j\in\sobo^{1,2}(\Omega;\R^2)$ by construction, cf.~the proof of \Cref{prop:EkelandSequence}. For the last term $\mathrm{III}$ we use the estimate \eqref{eq:EstimateNegSob3} for negative Sobolev spaces, namely we have
 	\begin{align*}
 		\mathrm{III} = \frac{1}{j} \|\triangle_{s,-h}(\varrho^2 \triangle_{s,h} v_j)\|_{\sobo^{-2,1}(\Omega;\R^2)}\leq \frac{1}{j} \|\varrho^2 \triangle_{s,h} v_j\|_{\lebe^1(\Omega;\R^2)} \leq \|\del_s v_j\|_{\lebe^1(\Omega;\R^2)}. 
 	\end{align*}
 	Collecting all the estimates leads to
 	\begin{align*}
 		\frac{1}{A_j j^2} \int_\Omega \abs{\varrho \A[\D (\triangle_{s,h}v_j)]}^2 \dx{x} \leq c, 
 	\end{align*}
 for a constant $c>0$ independent of $h>0$. Therefore, the family $(\triangle_{s,h}(\A[\D v_j]))_{h>0}$ is uniformly bounded in $\lebe^2(\ball_{r}(x_0);\R^{2\times2})$, implying the existence of the derivatives $\del_s \A[\D v_j]$ in $\lebe^2(\ball_{r}(x_0);\R^{2\times2})$ for $s\in\{1,2\}$. Thus, by Lemma \ref{lem:Cellipt2d} we conclude $\D^2 v_j\in\lebe^2(\ball_{r}(x_0);\R^{2\times2\times2})$. Since $x_0\in\Omega$ and $0<r<R<\mathrm{dist}(x_0,\del\Omega)$ were arbitrary, we have shown $v_j\in\sobo^{2,2}_\loc(\Omega;\R^2)$ and the proof is complete. 
\end{proof}
\begin{thm}[Differentiated Euler-Lagrange inequality]\label{thm:DiffEulerLagrangeInequality}
Under the assumptions of \Cref{lem:NonUniformSecEstimates}	it holds for all $k\in\{1,2\}$
	\begin{align*}
		\left\vert \int_\Omega \skalarProd{\partial_k\sigma_j}{\D \varphi}\dx{x}\right\vert =\left\vert \int_\Omega \skalarProd{\partial_k\sigma_j}{\A[\D \varphi]}\dx{x}\right\vert \leq \frac{1}{j}\|\varphi\|_{\sobo^{-1,1}(\Omega;\R^2)} \qquad \forall\ \varphi\in\sobo^{1,2}_0(\Omega;\R^2),.
	\end{align*}
where $(v_j)_{j\in\N}$ is the Ekeland-type approximation sequence. 
\end{thm}
\begin{proof}
	Since $\A[\sigma_j]=  \sigma_j$ we also have
	\begin{align}\label{eq:DerivativeSigma}
		\A[\partial_k\sigma_j]=\partial_k\A[\sigma_j]=\partial_k\sigma_j.
	\end{align}
	To conclude, we use $\del_s \varphi$ with $\varphi\in\hold^\infty_c(\Omega;\R^2)$ and $s\in\{1,2\}$ as a test-function in \eqref{eq:EulerLagrangeInequality}  and then apply an approximation argument like in \cite[Lemma 4.4]{Gmeineder}.
\end{proof}
\subsection{Uniform weighted second-order estimates}

As a  key ingredient in the proof of our main theorem \Cref{thm:MainTheorem} we will need  weighted second-order estimates for the Ekeland-type approximation sequence $(v_j)_{j\in\N}$, which are uniformly in $j\in\N$. Namely we have:

\begin{thm}[Weighted second-order estimates] Let $\A:\R^{2\times 2}\to\R^{2\times2}$ be an (orthogonal) projection such that the induced differential operator on $\R^2$ from $\R^2$ to $\R^{2\times2}$ given via $\Aop\coloneqq \A[\cdot\otimes \nabla]$ is $\C$-elliptic. Then there exists a constant $c=c(\Lambda, \A)>0$ such that
\begin{align}\label{eq:WeightedSecondOrder}
&\sum_{k=1}^2 \int_{\ball_{2r}(x_0)}\varrho^4 \skalarProd{\D^2f_j(\A[\D v_j])\partial_k\A[\D v_j]}{\partial_k\A[\D v_j]}\intd x\notag \\
&\hspace{0.5cm}\leq \frac{c(\Lambda,\A)}{r^2} \bigg[\int_{\ball_{2r}(x_0)} \abs{\A [\D v_j]}\dx{x}\\
&\hspace{2cm}+ \frac{1}{A_j j^2}\int_{\ball_{2r}(x_0)} (1+\abs{\A [\D v_j]}^2)\dx{x}
+ \bigg(1+\frac{r^2}{j}+\frac{r^3}{j} \bigg)\int_{\ball_{2r}(x_0)} \frac{\abs{v_j}}{r}\dx{x} \bigg]\notag,
\end{align}
where $\varrho\in\hold^\infty_c(\Omega;\R)$ is a localisation function with $\chi_{\ball_{r}(x_0)}\le\varrho\le\chi_{\ball_{2r}(x_0)}$, $\abs{\nabla  \varrho} \leq \frac2r$ and $\abs{\D^2\varrho}\le\frac{4}{r^2}$.
\end{thm}
\begin{proof} From \Cref{lem:NonUniformSecEstimates} we recall $v_j\in \sobo^{2,2}_{\loc}(\Omega;\R^2)$ and therefore, we can apply the product rule to obtain
	\begin{align*}
		\A[\D (\varrho^4 \del_k v_j)] = \A[\del_k v_j\otimes \nabla \varrho^4]+ \varrho^4 \A[\del_k \D v_j].
	\end{align*}

	\pagebreak \noindent This allows to rewrite 
	\begin{align}\label{eq:A+B}
		\int_\Omega &\varrho^4\skalarProd{\D^2 f_j(\A[\D v_j])\del_k \A[\D v_j]}{\partial_k\A[\D v_j]}\intd x\notag \\
		&= -\int_\Omega \skalarProd{\D^2 f_j(\A[\D v_j])\del_k \A[\D v_j]}{\A[\del_k v_j \otimes \nabla \varrho^4]}\intd x \notag \\
		&\hspace{2cm} +\int_\Omega \skalarProd{\D^2 f_j (\A[\D v_j])\del_k \A[\D v_j]}{ \A[\D(\varrho^4\del_k v_j)]}\dx{x}\notag \\
		&=-\int_{\Omega}\skalarProd{\partial_k\sigma_j}{\A[\del_k v_j \otimes \nabla \varrho^4]}\intd x+\int_{\Omega}\skalarProd{\partial_k\sigma_j}{\A[\D(\varrho^4\del_k v_j)]}\dx{x}\eqqcolon \mathrm{A} + \mathrm{B}.  
	\end{align}
	The second integral $\mathrm{B}$ can be handled immediately using the differentiated Euler-Lagrange inequality from \Cref{thm:DiffEulerLagrangeInequality}, once noticing that $\varrho^4 \del_k v_j\in \sobo^{1,2}_0(\Omega;\R^2)$. Therefore, we infer  
	\begin{align}\label{eq:estimateB}
		\abs{\mathrm{B}}
		&\leq \frac{1}{j}\|\varrho^4 \del_k v_j \|_{\sobo^{-1,1}(\Omega;\R^2)}\notag \\
		&\leq \frac{1}{j}\| \del_k (\varrho^4 v_j) -v_j \del_k \varrho^4\|_{\sobo^{-1,1}(\Omega;\R^2)}
		\leq \frac{1}{j} \bigg(1+\frac{8}{r}\bigg)\|v_j\|_{\lebe^1(\ball_{2r}(x_0);\R^2)}. 
	\end{align}
		The treatment of $\mathrm{A}$ is more subtle, as the expression $\del_k v_j \otimes \nabla \varrho^4$ is, in general, not a full a gradient. In view of \eqref{eq:DerivativeSigma} it holds
	\begin{align*}
		\mathrm{A} &= -\int_\Omega \skalarProd{\del_k \sigma_j}{\A[\del_k v_j\otimes \nabla \varrho^4]}\dx{x}
		=-\int_\Omega \skalarProd{\del_k \sigma_j}{\del_k v_j\otimes \nabla \varrho^4}\dx{x}.
	\end{align*}
Since $\A$ induces a $\mathbb{C}$-elliptic differential operator in two spatial dimensions, the representation \eqref{eq:complementarypart} yields
	\begin{align}\label{eq:Decomposition0}
		\del_k v_j \otimes \nabla \varrho^4  \notag
		&= (\D v_j)e_k\otimes \nabla \varrho^4\\
	&=\mathfrak{L}(\A[\D v_j])e_k \otimes \nabla \varrho^4 + \gamma(\D v_j) \mathfrak{G}e_k\otimes \nabla \varrho^4\notag\\
	&=\mathfrak{L}(\A[\D v_j])e_k \otimes \nabla \varrho^4 + \gamma(\D v_j) \D(\mathfrak{G}e_k \varrho^4).
	\end{align}
This allows to split $\mathrm{A}$ into two parts, namely
	\begin{align*}
		\abs{\mathrm{A}} 
		&\leq \left\vert \int_{\ball_{2r}(x_0)}\skalarProd{\del_k \sigma_j}{\mathfrak{L}(  \A[\D v_j])e_k\otimes \nabla \varrho^4}\dx{x}\right\vert + \left\vert \int_{\ball_{2r}(x_0)} \skalarProd{\del_k \sigma_j}{\gamma(\D v_j) \D(\mathfrak{G}e_k \varrho^4)}\dx{x}\right\vert 
		\eqqcolon \mathrm{A}_1 + \mathrm{A}_2.
	\end{align*}
	\noindent \emph{Ad $\mathrm{A}_1$.}
	Using the Cauchy-Schwarz inequality, which is applicable by \eqref{eq:BoundsBilinearform}, as well as the linearity of  $\mathfrak{L}\colon \R^{2\times2}\to \R^{2\times2}$, we obtain 
	\begin{align*}
		\mathrm{A}_1 &= \int_{\ball_{2r}(x_0)} \skalarProd{\D^2 f_j(\A[\D v_j])\del_k \A[\D v_j]}{\mathfrak{L}(\A[\D v_j])e_k \otimes \nabla\varrho^4}\dx{x} \\
		&\leq \frac{1}{2} \int_{\ball_{2r}(x_0)}  \skalarProd{\D^2 f_j (\A[\D v_j])\del_k \A[\D v_j]}{\del_k \A[\D v_j]}\dx{x} \\
		&\hspace{2cm} + \frac{1}{2}\int_{\ball_{2r}(x_0)}  \skalarProd{\D^2 f_j(\A[\D v_j])\mathfrak{L}(\A[\D v_j])e_k\otimes \nabla\varrho^4}{\mathfrak{L}(\A[\D v_j])e_k\otimes \nabla\varrho^4}\dx{x}.
	\end{align*}
	Now the first term can be absorbed into the left-hand side of \eqref{eq:A+B}, while the second can be controlled, because of the upper bound
	\begin{align*}
	\skalarProd{\D^2 f_j (P)\xi}{\xi} \leq \Lambda \frac{\abs{\xi}^2}{(1+\abs{P}^2)^{\frac{1}{2}}} + \frac{1}{A_j j^2}\abs{\xi}^2.
	\end{align*}	
	Therefore, we conclude 
	\begin{align*}
		\int_{\ball_{2r}(x_0)}  &\skalarProd{\D^2 f_j(\A[\D v_j])\mathfrak{L}(\A[\D v_j])e_k\otimes \nabla\varrho^4}{\mathfrak{L}(\A[\D v_j])e_k\otimes \nabla\varrho^4}\dx{x}\\
		&\leq \frac{\Lambda}{2}\int_{\ball_{2r}(x_0)} \frac{\abs{\mathfrak{L}(\A[\D v_j])e_k\otimes \nabla\varrho^4}^2}{(1+\abs{\A[\D v_j] }^2)^{\frac{1}{2}}}\dx{x}
		+ \frac{1}{2A_j j^2 }\int_{\ball_{2r}(x_0)}  \abs{\mathfrak{L}(\A[\D v_j])e_k\otimes \nabla\varrho^4}^2\dx{x}\\
		&\leq \frac{c(\Lambda, \A)}{r^2} \int_{\ball_{2r}(x_0)}  \abs{\A[\D v_j]}\dx{x} + \frac{c(\A)}{ A_j j^2 r^2} \int_{\ball_{2r}(x_0)}  (1+\abs{\A[\D v_j]}^2) \dx{x}. 
	\end{align*}
	In view of \Cref{prop:EkelandSequence} both terms can be estimated.

\emph{Ad $\mathrm{A}_2$.} The second term $\mathrm{A}_2$ deserves a more thorough analysis. In view of the linearity of $\gamma$ (so that $\gamma(\cdot)=\skalarProd{Q}{\cdot}$ for a suitable $Q\in\R^{2\times2}$) and since $v_j\in \sobo^{2,2}_{\loc}(\Omega;\R^2)$ by \Cref{lem:NonUniformSecEstimates}, we can integrate by parts:
\begin{align}
 &\int_{\ball_{2r}(x_0)} \skalarProd{\del_k \sigma_j}{\gamma(\D v_j) \D(\mathfrak{G}e_k \varrho^4)}\dx{x}\notag \\ 
 &= -\int_{\ball_{2r}(x_0)}\skalarProd{\sigma_j}{\skalarProd{Q}{\partial_k\D v_j}\D(\mathfrak{G}e_k \varrho^4)}\dx{x}
 -\int_{\ball_{2r}(x_0)}\skalarProd{\sigma_j}{\gamma(\D v_j)\D(\mathfrak{G}e_k \partial_k\varrho^4)} \dx{x}\\
 & \eqqcolon \mathrm{A}_{2_{a}} + \mathrm{A}_{2_{b}} \notag
\end{align}

\noindent To treat the second integral $\mathrm{A}_{2_{b}}$, we first utilise the product rule to deduce
\begin{align}
 \gamma(\D v_j)\D(\mathfrak{G}e_k \partial_k\varrho^4) &= \D(\gamma(\D v_j)\mathfrak{G}e_k \partial_k\varrho^4) - \mathfrak{G}e_k \partial_k\varrho^4 \otimes \nabla(\gamma(\D v_j))\notag \\
 &= \D(\gamma(\D v_j)\mathfrak{G}e_k \partial_k\varrho^4) - \mathfrak{G}e_k \partial_k\varrho^4 \otimes \begin{pmatrix}\skalarProd{Q}{\partial_1\D v_j}\\[1ex]\skalarProd{Q}{\partial_2\D v_j} \end{pmatrix},
\end{align}
where in the last step we have used again the linearity of  $\gamma$. Hence, in view of Lemma \ref{lem:Cellipt2d} there exist constant coefficients linear maps $\mathcal{L}_\ell:\R^{2\times2}\to\R^{2\times2}$ such that we have
\begin{align}\label{eq:intparts6573}
 \gamma(\D v_j)\D(\mathfrak{G}e_k \partial_k\varrho^4) &= \D(\gamma(\D v_j)\mathfrak{G}e_k \partial_k\varrho^4) -\partial_k\varrho^4\sum_{\ell=1}^2 \partial_\ell\left(\mathcal{L}_{\ell}\big(\A[\D v_j]\big)\right).\end{align}
 First note, that the first term can be treated using the differentiated Euler-Lagrange inequality from \Cref{thm:DiffEulerLagrangeInequality}, since we have a full gradient in the second argument of the scalar product. Thus, by the boundedness of $\varrho$ and the linearity of $\gamma\colon \R^{2\times2}\to\R$, we obtain 
	\begin{align*}
		&\abs{\int_{\ball_{2r}(x_0)}\skalarProd{\sigma_j}{\D(\gamma(\D v_j)\mathfrak{G}e_k \partial_k\varrho^4)}\dx{x}}\\
		&\hspace{1cm}\le \frac1j\norm{\gamma(\D v_j)\mathfrak{G}e_k \partial_k\varrho^4}_{\sobo^{-1,1}({\ball_{2r}(x_0)};\R^{2})} \\
		&\hspace{1cm} =\frac1j\norm{\sum_{\ell,m=1}^2Q^{(\ell m)}\partial_m v_j^{(\ell)}\mathfrak{G}e_k \partial_k\varrho^4}_{\sobo^{-1,1}({\ball_{2r}(x_0)};\R^{2})}\\
		&\hspace{1cm}\le \frac1j\sum_{m=1}^2\norm{\partial_m\left((Q^\top v_j)^{(m)}\mathfrak{G}e_k \partial_k\varrho^4\right)-(Q^\top v_j)^{(m)}\mathfrak{G}e_k \partial_m\partial_k\varrho^4}_{\sobo^{-1,1}({\ball_{2r}(x_0)};\R^{2})}\\
		&\hspace{1cm}\leq \frac{c(\A)}{j}\left(\frac1r+\frac{1}{r^2}\right)\|v_j\|_{\lebe^1({\ball_{2r}(x_0)};\R^2)}.  
	\end{align*}
Note that, by integration by parts, the second term arising from \eqref{eq:intparts6573} satisfies:

 \begin{align}\label{eq:badpartA2b}
  &\abs{\sum_{\ell=1}^2 \int_{\ball_{2r}(x_0)}\partial_k\varrho^4\skalarProd{\sigma_j}{\partial_\ell\left(\mathcal{L}_{\ell}\big(\A[\D v_j]\big)\right)}\dx{x}} \notag \\
   &\hspace{1cm}\le \abs{\sum_{\ell=1}^2 \int_{\ball_{2r}(x_0)}\partial_k\varrho^4\skalarProd{\partial_{\ell}\sigma_j}{\mathcal{L}_{\ell}\big(\A[\D v_j]\big)}\dx{x}} \\
    &\hspace{2cm} +c(\A) \int_{\ball_{2r}(x_0)} \abs{\sigma_j}\abs{\A[\D v_j]}\abs{\D^2 \varrho^4}\dx{x} \eqqcolon\mathrm{A}_{2_{b}}' + \mathrm{A}_{2_{b}}'' .\notag
 \end{align}
 For the second integral $\mathrm{A}_{2_{b}}''$ we immediately obtain by \eqref{eq:sigma}
 \begin{align}
  \mathrm{A}_{2_{b}}'' \le \frac{c(\A)}{r^2}\int_{\ball_{2r}(x_0)}\abs{\A[\D v_j]}\dx{x} + \frac{c(\A)}{A_jj^2r^2}\int_{\ball_{2r}(x_0)}\abs{\A[\D v_j]}^2\dx{x},
 \end{align}
and in view of \Cref{prop:EkelandSequence} both terms can be estimated. Furthermore, by \eqref{eq:BoundsBilinearform}, we can apply the Cauchy-Schwarz inequality to estimate the first integral $\mathrm{A}_{2_{b}}'$:

\pagebreak
\begin{align*}
  &\sum_{\ell=1}^2\int_{\ball_{2r}(x_0)}\partial_k\varrho^4\skalarProd{\partial_\ell\sigma_j}{\mathcal{L}_{\ell}\big(\A[\D v_j]\big)}\dx{x} \\
  &\hspace{1cm}= \sum_{\ell=1}^2 \int_{\ball_{2r}(x_0)}\skalarProd{\D^2f_j(\A[\D v_j])\varrho^2\partial_\ell\A[\D v_j]}{4\varrho\partial_k\varrho\mathcal{L}_{\ell}\big(\A[\D v_j]\big)}\dx{x}\notag \\
  &\hspace{1cm}\le \epsilon\sum_{\ell=1}^2 \int_{\ball_{2r}(x_0)}\skalarProd{\D^2f_j(\A[\D v_j])\varrho^2\partial_\ell\A[\D v_j]}{\varrho^2\partial_\ell\A[\D v_j]}\dx{x}\\
  &\hspace{2cm}+\frac{16}{\epsilon}\sum_{\ell=1}^2 \int_{\ball_{2r}(x_0)}\skalarProd{\D^2f_j(\A[\D v_j])\varrho\partial_k\varrho\mathcal{L}_{\ell}\big(\A[\D v_j]\big)}{\varrho\partial_k\varrho\mathcal{L}_{\ell}\big(\A[\D v_j]\big)}\dx{x}. \notag
\end{align*}
Since we sum over $k$ in \eqref{eq:WeightedSecondOrder}, the first term on the right-hand side can always be absorbed into the left-hand side of \eqref{eq:WeightedSecondOrder} for $\epsilon=\epsilon(\A)>0$ sufficiently small,  while the second can be controlled, due to the upper bound
	\begin{align*}
	\skalarProd{\D^2 f_j (P)\xi}{\xi} \leq \Lambda \frac{\abs{\xi}^2}{(1+\abs{P}^2)^{\frac{1}{2}}} + \frac{1}{A_j j^2}\abs{\xi}^2.
	\end{align*}	
	Therefore, we conclude 
\begin{align}
 \frac{16}{\epsilon}\int_{\ball_{2r}(x_0)}&\skalarProd{\D^2f_j(\A[\D v_j])\varrho\partial_k\varrho\mathcal{L}_{\ell}\big(\A[\D v_j]\big)}{\varrho\partial_k\varrho\mathcal{L}_{\ell}\big(\A[\D v_j]\big)}\dx{x} \notag\\
 &\le \frac{c(\Lambda,\A)}{r^2}\int_{\ball_{2r}(x_0)}\frac{\abs{\A[\D v_j]}^2}{(1+\abs{\A[\D v_j]}^2)^{\frac12}}\dx{x}+\frac{c(\A)}{A_jj^2r^2}\int_{\ball_{2r}(x_0)}\abs{\A[\D v_j]}^2\dx{x} \notag\\
 &\le \frac{c(\Lambda,\A)}{r^2}\int_{\ball_{2r}(x_0)}\abs{\A[\D v_j]}\dx{x}+\frac{c(\A)}{A_jj^2r^2}\int_{\ball_{2r}(x_0)}(1+\abs{\A[\D v_j]}^2)\dx{x}
\end{align}
In view of \Cref{prop:EkelandSequence} both terms can again be estimated, and the estimation of $\mathrm{A}_{2_{b}}$ is now complete.

\noindent It remains to estimate the first integral $\mathrm{A}_{2_{a}}$. However, in view of Lemma \ref{lem:Cellipt2d} there exist constant coefficients linear maps $\widetilde{\mathcal{L}_\ell}, \widehat{\mathcal{L}_{\ell}}:\R^{2\times2}\to\R^{2\times2}$  such that for $\mathrm{A}_{2_{a}}$ we obtain:
\begin{align}
\mathrm{A}_{2_{a}} = \int_{\ball_{2r}(x_0)}\skalarProd{\sigma_j}{\sum_{\ell=1}^2 \partial_\ell\left(\widetilde{\mathcal{L}_{\ell}}\big(\A[\D v_j]\big)\right)}\varrho^3\partial_1\varrho + \skalarProd{\sigma_j}{\sum_{\ell=1}^2 \partial_\ell\left(\widehat{\mathcal{L}_{\ell}}\big(\A[\D v_j]\big)\right)}\varrho^3\partial_2\varrho\dx{x}.
\end{align}
We again integrate by parts and argue as in the estimation of $\mathrm{A}_{2_{b}}$, cf. \eqref{eq:badpartA2b}. Summarising, we infer 
 \begin{align*}
 	\sum_{k=1}^2\int_{{\ball_{2r}(x_0)}} &\varrho^4\skalarProd{\D^2 f_j(\A[\D v_j])\del_k \A[\D v_j]}{\partial_k\A[\D v_j]}\intd x\\
 	&\leq \frac{c(\A)}{j}\bigg(1+\frac{1}{r}+\frac{1}{r^2}\bigg)\|v_j\|_{\lebe^1({\ball_{2r}(x_0)};\R^2)} + \frac{c(\Lambda,\A)}{r^2} \,\|\A[\D v_j]\|_{\lebe^1({\ball_{2r}(x_0)};\R^2)} \\
 	&\hspace{2cm}+ \frac{c(\A)}{A_j j^2r^2}\int_{{\ball_{2r}(x_0)}}(1+\abs{\A[\D v_j]}^2)\dx{x}.
 \end{align*}
 Since $\A[\cdot\otimes\nabla]$ is a first-order $\C$-elliptic differential operator we use the Poincar\'e-type inequality from  \cite[Lemma 2.12]{GmRa1} together with an approximation argument to estimate the $\lebe^1$-norm of $v_j$ by the $\lebe^1$-norm of $\A[\D v_j]$. More precisely, taking into account \eqref{eq:EstimateMollification}, we deduce
 \begin{align}\label{eq:L1-bound_Ekeland_C_elliptic}
 	\|v_j\|_{\lebe^1(\Omega;\R^2)} &\leq \|v_j-u_j^{\del\Omega}\|_{\lebe^1(\Omega;\R^2)} + \|u_j^{\del\Omega} \|_{\lebe^1(\Omega;\R^2)} \notag \\
 		&\leq c(\Omega)\,\| \Aop v_j-\Aop u_j^{\del\Omega}\|_{\lebe^1(\Omega;\R^{2\times 2})}+\|u_j^{\del\Omega}\|_{\lebe^1(\Omega;\R^2)} \notag \\
 		&\leq c(\Omega)\bigg(\|\Aop v_j\|_{\lebe^1(\Omega;\R^2)} + \|u_j^{\del\Omega}\|_{\sobo^{\Aop,1}(\Omega)}\bigg)\\
 		&\leq c(f,u_0,\Omega) \left[\frac{1}{c_1}\left(\inf_{\mathscr{Di}_{u_0}} \F[-;\Omega] + \frac{2}{j^2}\right)+\frac{1}{j^2}\right]\notag .
\end{align}

\pagebreak Finally, collecting all the estimates from above results in
 \begin{align*}
  	&\int_{\ball_{2r}(x_0)} \varrho^4\skalarProd{\D^2 f_j(\A[\D v_j])\del_k \A[\D v_j]}{\partial_k\A[\D v_j]}\intd x\\
  	&\hspace{0.5cm}\leq \frac{c(\Lambda,\A)}{r^2} \bigg[\int_{\ball_{2r}(x_0)} \abs{\A [\D v_j]}\dx{x} + \frac{1}{A_j j^2}\int_{\ball_{2r}(x_0)} (1+\abs{\A [\D v_j]}^2)\dx{x}\\ 
  	&\hspace{2cm}+ \bigg(\frac{r^3}{j}+\frac{r^2}{j} +\frac{r}{j}\bigg)\int_{\ball_{2r}(x_0)} \frac{\abs{v_j}}{r}\dx{x} \bigg]\\
  	&\leq \frac{c(\Lambda,\A)}{r^2} \bigg[\int_{\ball_{2r}(x_0)} \abs{\A [\D v_j]}\dx{x} + \frac{1}{A_j j^2}\int_{\ball_{2r}(x_0)} (1+\abs{\A [\D v_j]}^2)\dx{x}\\ 
  	&\hspace{2cm}+ \bigg(1+\frac{r^2}{j}+\frac{r^3}{j} \bigg)\int_{\ball_{2r}(x_0)} \frac{\abs{v_j}}{r}\dx{x} \bigg],
 \end{align*}
where we have used the inequality $\frac{r}{j} \leq 1+\frac{r^2}{j}$ in the last line. This yields the desired estimate and finishes the proof. 
\end{proof}

\begin{rem}
	Assuming in addition, that the integrand $f$ is $\mu$-elliptic in the sense of \eqref{eq:MuEllipticity} we can bound \eqref{eq:WeightedSecondOrder} from below to obtain the following weighted second-order estimates: 
	\begin{align}\label{eq:WeightedSecondOrder_MuElliptic}
		&\int_{\ball_{2r}(x_0)} \varrho^4 (1+\abs{\A[\D v_j]}^2)^{-\frac{\mu}{2}} \abs{\D\A[\D v_j]}^2\dx{x} \notag \\
		&\hspace{0.5cm}\leq \sum_{k=1}^2\int_{\ball_{2r}(x_0)} \varrho^4\skalarProd{\D^2 f_j(\A[\D v_j])\del_k \A[\D v_j]}{\partial_k\A[\D v_j]}\intd x \notag \\
		&\hspace{0.5cm}\leq \frac{c(\Lambda,\A)}{r^2} \bigg[\int_{\ball_{2r}(x_0)} \abs{\A [\D v_j]}\dx{x}\\
		&\hspace{2cm}+ \frac{1}{A_j j^2}\int_{\ball_{2r}(x_0)} (1+\abs{\A [\D v_j]}^2)\dx{x}
		+ \bigg(1+\frac{r^2}{j}+\frac{r^3}{j} \bigg)\int_{\ball_{2r}(x_0)} \frac{\abs{v_j}}{r}\dx{x} \bigg]\notag\\
		&\hspace{0.5cm}\leq 
		\frac{c}{r^2}\left(1+\frac{r^2}{j}+\frac{r^3}{j}\right)\left[\frac{1}{c_1}\left(\inf_{\mathscr{D}_{u_0}} \F[-;\Omega] + \frac{2}{j^2}\right)+\frac{1}{j^2}\right] + \frac{c}{r^2 j^2}\notag ,
	\end{align}
	for a constant $c=c(c_1,c_2,\lambda,\Lambda,\A,u_0,\Omega)>0$.
\end{rem}

\begin{rem}\label{rem:dev1}
	The previous proof strongly relies on the structure of first-order $\C$-elliptic differential operators in two dimensions. Indeed, the presented arguments \eqref{eq:A+B} and \eqref{eq:estimateB} are applicable for any first-order differential operator $\A[\cdot\otimes\nabla]$ in all dimensions. However, we were able to estimate the remaining term $\mathrm{A}$ by exploiting the decomposition \eqref{eq:complementarypart}, which relies on the $\C$-ellipticity in two dimensions. Indeed, whenever such a decomposition persists, we can follow the above arguments. Let us demonstrate it for the deviatoric (trace-free) gradient in all dimensions. In this framework, $\A=\dev:\R^{n\times n}\to \R^{n\times n}$ is given by
	\[
	\dev P\coloneqq P -\frac{\tr P}{n}\,\bbone_n,
	\]
	so that $\mathfrak{L}=\mathrm{id}_{\R^{n\times n}}$, $\gamma(P)=\displaystyle\frac{\skalarProd{P}{\bbone_n}}{n}$ and $\mathfrak{G}=\bbone_n$ for $P\in\R^{n\times n}$. Hence, we can repeat the above steps for the particular choice $\A=\dev$ in all dimensions. Note, that $\dev\D$ is a $\C$-elliptic differential operator in all dimensions whose complementary part is always one-dimensional. It is precisely this structure, that was used in the proof for $\C$-elliptic differential operators in $2$d which possess a one-dimensional so called \emph{almost complementary part}, cf. \cite[Proposition 4.1]{GLN1}. This situation changes drastically if one focuses on the trace-free symmetric gradient. Note, that the trace-free symmetric gradient is not $\C$-elliptic in $2$d and already in $3$d the complementary part of the trace-free symmetric gradient is no longer one-dimensional. However, the symmetric gradient $\varepsilon=\sym\D$ case in all dimensions was treated in \cite{GK,Gmeineder,BEG}. 
\end{rem}

\subsection{Korn inequalities in Orlicz spaces for $\C$-elliptic differential operators} \label{sec:Korn_in_Orlicz} 
In this section we investigate into Korn-type inequalities on the scale of Orlicz-spaces. We begin by mentioning several known results.  In \cite{BreitDiening_KornOrlicz} \textsc{Breit} and \textsc{Diening} prove that an inequality of the form  
\begin{equation}\label{eq:Korn_Remark}
	\|\D u\|_{\lebe^\Phi(\Omega;\R^{n\times n})}\lesssim \|\sym \D u\|_{\lebe^\Phi(\Omega;\R^{n\times n}_{\rm sym})}
\end{equation}
can only hold for all $u\in\hold^\infty_c(\R^n;\R^{n})$ if and only if $\Phi\in \nabla_2\cap\Delta_2$, which was extended to more general elliptic operators by \textsc{Conti} and \textsc{Gmeineder} \cite{ContiGmeineder}. This result is somehow natural as $\Phi\in\Delta_2$ and $\Phi\in\nabla_2$ loosely speaking mean that the norm $\|\cdot\|_{\lebe^\Phi}$ is not \emph{too close} to $\|\cdot\|_{\lebe^\infty}$ and $\|\cdot\|_{\lebe^1}$ respectively, for which the Korn inequality fails. Moreover, it is possible to weaken the assumption $\Phi\in\nabla_2\cap \Delta_2$ by replacing the norm on the right-hand side of \eqref{eq:Korn_Remark} by a slightly weaker Orlicz norm due to results of \textsc{Cianchi} \cite{Cianchi} for the symmetric gradient.
Before establishing the needed Korn-type inequalities in Orlicz spaces, we catch up with the proof of Lemma \ref{lem:Cellipt2d}:
\begin{proof}[Proof of Lemma \ref{lem:Cellipt2d}]
 Our proof relies on the special structure of $\C$-elliptic differential operators on $\R^2$, see \eqref{eq:complementarypart} and is motivated by construction of the corresponding linear maps $\mathfrak{L}$ and $\gamma$, as well as of the matrix $\mathfrak{G}\in\mathrm{GL}(2)$, cf. the proof of \cite[Prop. 4.1]{GLN1}. We distinguish the only two possible cases:
 
 If $\dim{\A[\R^{2\times2}]}=4$ then $\{\e_i\otimes_\Aop\e_j\}_{i,j=1,2}$ form a basis of $\A[\R^{2\times2}]$ and the linear map $\mathfrak{L}$ can, e.g., be defined via the action
 \begin{align*}
  \mathfrak{L}(\e_i\otimes_\Aop\e_j)\coloneqq \e_i\otimes\e_j-\bbone_2, \qquad \text{for } i,j=1,2.
 \end{align*}
With this choice we obtain for $u\in\hold^\infty_c(\R^2;\R^2)$:
\begin{align*}
 \mathfrak{L}(\A[\D u])=\begin{pmatrix} 
                       -(\partial_1 u_2 +\partial_2 u_1 + \partial_2 u_2) & \partial_1 u_2 \\ \partial_2 u_1 & -(\partial_1u_1+\partial_1u_2+\partial_2 u_1)
                        \end{pmatrix}.
\end{align*}
Hence, we could already express all the entries of the gradient $\D u$ by a linear combination of the entries of $\mathfrak{L}(\A[\D u])$:
\begin{align*}
 \partial_1u_1=  -\mathfrak{L}(\A[\D u])^{(22)}-\mathfrak{L}(\A[\D u])^{(12)} - \mathfrak{L}(\A[\D u])^{(21)}, \quad\ldots
\end{align*}
and could even take $d=1$ in this case. But then also all the entries of $\D^2 u$ can be expressed as a linear combination of the entries of $\D\mathfrak{L}(\A[\D u])$.

The more sophisticated case is when $\dim{\A[\R^{2\times2}]}=3$. Then, like in the proof of \cite[Prop. 4.1]{GLN1}, without loss of generality, we find coefficients $a_{11}, a_{12}, a_{22}$ not all equal to zero such that
\begin{align*}
 \e_2\otimes_{\Aop}\e_1=a_{11}\e_1\otimes_\Aop \e_1 + a_{12}\e_1\otimes_{\Aop}\e_2+a_{22}\e_2\otimes_{\Aop}\e_2
\end{align*}
and, cf. \cite[Eq. (4.8)]{GLN1},
\begin{align}\label{eq:detG}
 a_{11}a_{22}+a_{12}\neq0.
\end{align}
In this case, the linear map $\mathfrak{L}$ can be chosen in the following way, cf. \cite[Eq. (4.12) f]{GLN1}:
\begin{align*}
 \mathfrak{L}(\e_i\otimes_\Aop\e_j)=\begin{cases}
                                     \e_i\otimes \e_j & \text{if } (i,j)\neq(2,1),\\[1ex]
                                     \begin{pmatrix}
                                      a_{11} & a_{12} \\ 0 & a_{22}
                                     \end{pmatrix} & \text{if } (i,j)=(2,1).
                                    \end{cases}
\end{align*}
Then
\begin{align*}
 \mathfrak{L}(\A[\D u])=\begin{pmatrix}
                       \partial_1 u_1 + a_{11}\partial_2 u_1 & \partial_1 u_2 + a_{12}\partial_2 u_1 \\ 0 & \partial_2u_2 + a_{22}\partial_2 u_1  
                        \end{pmatrix},
\end{align*}
so that in general we cannot express all the entries of the gradient $\D u$ by a linear combination of the entries of this $\mathfrak{L}(\A[\D u])$. However, differentiating all the entries we obtain
\begin{align*}
 \begin{pmatrix}
  \partial_1 \mathfrak{L}(\A[\D u])^{(11)}\\
  \partial_2 \mathfrak{L}(\A[\D u])^{(11)}\\
  \partial_1 \mathfrak{L}(\A[\D u])^{(12)}\\
  \partial_2 \mathfrak{L}(\A[\D u])^{(12)}\\
  \partial_1 \mathfrak{L}(\A[\D u])^{(22)}\\
  \partial_2 \mathfrak{L}(\A[\D u])^{(22)}
 \end{pmatrix}
 =
 \begin{pmatrix}
  1 & a_{11} & 0 & 0 & 0 & 0\\
  0 & 1 & a_{11} & 0 & 0 & 0\\
  0 & a_{12} & 0 & 1 & 0 & 0\\
  0 & 0 & a_{12} & 0 & 1 & 0\\
  0 & a_{22} & 0 & 0 & 1 & 0\\
  0 & 0 & a_{22} & 0 & 0 & 1
 \end{pmatrix}
 \begin{pmatrix}
  \partial_1^2 u_1\\
  \partial_1\partial_2 u_1 \\
  \partial_2^2u_1\\
  \partial_1^2 u_2\\
  \partial_1\partial_2 u_2 \\
  \partial_2^2u_2
 \end{pmatrix}
\end{align*}
since the determinant of the appearing $6\times6$-matrix is $-(a_{11}a_{22}+a_{12})\neq0$ by \eqref{eq:detG}, we can express all the entries of $\D^2 u$ by a linear combination of the entries of $\D\mathfrak{L}(\A[\D u])$. This completes the proof of Lemma \ref{lem:Cellipt2d}.
 \end{proof}
\begin{rem}\label{rem:dev2}
 The structure from \Cref{lem:Cellipt2d} also applies to the trace-free gradient in all dimensions. More precisely, all the entries of the second derivative of a vector field can be expressed by a linear combination of derivatives of $\Aop=\dev[\cdot\otimes\nabla]$, i.e, in Lemma \ref{lem:Cellipt} we have for this particular choice also $d=2$. Indeed, since this operator only applies on the diagonal elements, the off-diagonal entries remain the same:
 \[
  (\D u)^{(ij)}=(\dev \D u)^{(ij)} \qquad \text{for }i\neq j \in\{1,\ldots,n\}.
 \]
Furthermore, we note that the diagonal elements of $\D u$ cannot be expressed as linear combinations of the entries of the elements of $\dev \D u$. However, this holds true if we increase the order of differentiation:
\begin{align*}
 \partial_i(\dev\D u)^{(ii)}&=\partial_i\left(\partial_i u_i -\frac1n\sum_{j=1}^n\partial_ju_j\right)=\frac{n-1}{n}\partial_i\partial_i u_i -\frac1n\sum_{\stackrel{j=1}{j\neq i}}^n\partial_i\partial_j u_j\\
 &= \frac{n-1}{n}\partial_i(\D u)^{(ii)}-\frac1n\sum_{\stackrel{j=1}{j\neq i}}^n\partial_j(\dev \D u)^{(ji)}.
\end{align*}
Hence, we found a first-order differential operator $\mathbb{L}$ such that
\[
 \D^2 = \mathbb{L}\circ \dev[\cdot\otimes\nabla].
\]
Thus, we can follow the arguments below also in case of the deviatoric gradient in all dimensions and first establish the necessary Korn-type inequalities in Orlicz spaces. In conclusion, we obtain a similar statement as in our main theorem \ref{thm:MainTheorem} for the deviatoric gradient in all dimensions.

\end{rem}

As an immediate consequence we can characterise the nullspace of $\A[\cdot\otimes\nabla]$:
\begin{cor}\label{cor:kernel}
 Let the projection $\A:\R^{2\times 2}\to\R^{2\times2}$ induce a $\C$-elliptic differential operator on $\R^2$ by $\Aop\coloneqq\A[\cdot\otimes\nabla]$. Then $\Aop u = 0$ if and only if
 \begin{align*}
  u(x)=Bx +b \quad \text{with } B\in \ker \A \text{ and } b\in\R^2. \end{align*}
\end{cor}
\begin{proof}
 It follows directly from Lemma \ref{lem:Cellipt2d}.
\end{proof}

 As a next step, we recall the following statement from  {\cite[Theorem 5.1]{Stephan}}, which is is a generalisation of \cite[Theorem 3.1]{Cianchi}.
\begin{thm}[Korn-type inequality in Orlicz spaces]
	Let $\Omega\subset\R^n$, $n\geq 2$ be an open and bounded set and $\Aop$ be a homogeneous, first-order elliptic differential operator of the form \eqref{eq:DifferentialOperator}. Moreover, let $\Psi$ and $\Phi$ be Young functions such that 
	\begin{equation}\label{eq:CianchiAssumptionYoungFunctions}
		t \int_0^t \frac{\Phi(s)}{s^2}\dx{s} \leq \Psi(ct) \quad\mbox{and}\quad t\int_0^t \frac{\Psi^\ast(s)}{s^2}\dx{s}\leq \Phi^\ast(ct),
	\end{equation}
	hold for all $t\geq 0$ and some $c>0$. Then there exist a constant $C>0$ such that 
	\begin{equation}\label{eq:Korn_Type_Orlicz_Paul}
		\int_\Omega \Phi(\abs{\D u})\dx{x} \lesssim \int_\Omega \Psi(C\abs{\Aop u})\dx{x}\quad\mbox{for all}\quad u\in \sobo^{1,\Psi}_0(\Omega;V).
	\end{equation}
\end{thm}
\noindent As a matter of fact, modular estimates in Orlicz spaces are always stronger then Luxemburg norm. Therefore, in view of \eqref{eq:Korn_Type_Orlicz_Paul}, we obtain for $u\in \sobo^{1,\Psi}(\Omega;V)$ the estimate
\begin{equation}\label{eq:Korn_Type_Orlicz_Paul_Norm}
	\|\D u \|_{\lebe^\Phi(\Omega;\Lin(\R^n;V))} \leq C\, \|\Aop u\|_{\lebe^\Psi(\Omega;W)}.
\end{equation}
We now aim to prove a Korn-type inequality for balls on the scale of Orlicz spaces. Towards this aim, we first need to derive a Poincar\'{e}-type inequality in Orlicz spaces, which will we be the content of the following theorem:

\begin{thm}[Poincar\'e-type inequality]
Let $r>0$, $x_0\in\R^n$ and $\ball_r(x_0)\subset\R^n$ be a ball, $\Aop$ be a $\C$-elliptic operator of the form \eqref{eq:DifferentialOperator} and $\Psi$  be a Young function. Then, for every $u\in \sobo^{\Aop,\Psi}(\ball_r(x_0))$ there exists a constant $c=c(r,n,\A)>0$ such that
\begin{equation}\label{eq:PoincareOrlicz}
	\inf_{\rfrak\in\ker(\Aop)}\|u-\rfrak\|_{\L^\Psi(\ball_r(x_0);V)} \leq c\,\omega_n r\,\|\Aop u\|_{\L^\Psi(\ball_r(x_0);W)}. 
\end{equation}
\end{thm}
\begin{proof}
	For the sake of simplicity we abbreviate $\ball\coloneqq \ball_r(x_0)$. The result is an immediate consequence from \cite[Proposition 3.8.]{DieningGmeineder}, namely there exists a constant $c=c(\ball,\Aop)>0$ such that for $u\in \sobo^{\Aop,1}(\ball)$ there holds the point-wise estimate
	\begin{equation}\label{eq:PointwiseInequalityPoincare}
		\abs{(u-\Pi^\ball_{\Aop} u)(x)} \leq c\int_\ball \abs{x-y}^{1-n} \abs{\Aop u(y)}\dx{y}\quad\mbox{for  }\Leb^n\mbox{-a.e. } x\in\ball,
	\end{equation}
	where $\Pi^\ball_{\Aop}$ denotes a suitable linear projection onto $\ker(\Aop)$, cf. \cite[Prop. 3.3]{DieningGmeineder}. As a next step we define for $x\in\ball$ a Borel measure $\mu_x\colon\mathscr{B}(\ball)\to [0,\infty)$ via
	\begin{equation*}
		\mu_x (A)\coloneqq c\int_\ball \abs{x-y}^{1-n} \dx{y}\quad\mbox{for}\quad A\in\mathscr{B}(\ball)
	\end{equation*}
	and set $m_x\coloneqq \mu_x(\ball)$. Now, since $\abs{x-y}\leq \diam(\ball) = 2r$ for all $x,y\in \ball$, we can estimate
	\begin{equation}\label{eq:LowerBoundMeasure_Poincare}
		m_x =c \int_\ball \frac{\dx{y}}{\abs{x-y}^{n-1}}\geq c\frac{\Leb^n(\ball)}{(\diam(\ball))^{n-1}} = c\,\omega_n\, r\quad \mbox{for}\quad \Leb^n\mbox{-a.e. }x\in\Omega. 
	\end{equation}
	Dividing \eqref{eq:PointwiseInequalityPoincare} by $m_x$ and $\lambda\coloneqq \|\Aop u\|_{\L^\Psi(\ball;W)}$ and applying the Jensen inequality, we obtain for $\Leb^n$-a.e. $x\in\Omega$
	\begin{align*}
		\Psi\left(\frac{\abs{(u-\Pi^\ball_{\Aop} u)(x)}}{m_x\lambda}\right) 
		\leq \Psi\left(\, \dashint_{\ball}  \frac{\abs{\Aop u(y)}}{\lambda}\dx{\mu_x(y)}\right)
		\leq \dashint_\ball\Psi\left(\frac{\abs{\Aop u(y)}}{\lambda}\right) \dx{\mu_x(y)}
	\end{align*}
	Therefore, integrating the above inequality over $\ball$ with respect to the $x$-variable, in conjunction with the Tonelli theorem, results in 
	\begin{align*}
		\int_\ball \Psi\left(\frac{\abs{(u-\Pi^\ball_{\Aop} u)(x)}}{m_x\lambda}\right) \dx{x} 
		&\leq \int_\ball\,\dashint_\ball \Psi\left(\frac{\abs{\Aop u(y)}}{\lambda}\right) \dx{\mu_x(y)}\dx{x}\\
		&= \int_\ball  \Psi\left(\frac{\abs{\Aop u(y)}}{\lambda}\right)\int_\ball\frac{c}{\abs{x-y}^{n-1}}\frac{1}{m_x}\dx{y}\dx{x}\\
		&= \int_\ball \Psi\left(\frac{\abs{\Aop u(y)}}{\|\Aop u\|_{\L^{\Psi}(\ball;W)}}\right)\dx{y}\\
		&\leq 1,
	\end{align*}
	where the last inequality follows from the fact that $u\in\sobo^{\Aop,\Psi}(\ball)$. Using the lower bound \eqref{eq:LowerBoundMeasure_Poincare} together with the definition of the Luxemburg norm gives 
	\begin{align}\label{eq:Poincare_Projection}
		\|u-\Pi^\ball_{\Aop} u\|_{\L^\Psi(\ball;V)} \leq c\, \omega_n\, r \|\Aop u\|_{\L^\Psi(\ball;W)}, 
	\end{align}
	where $c=c(\Aop,\ball)>0$. This leads to the desired inequality \eqref{eq:PoincareOrlicz}, since we can always bound the left hand side by the infimum taken over all $\rfrak\in\ker(\Aop)$, thus the proof is complete.
	\end{proof}
As a side remark, we refer to \cite[Proposition 4.2]{GmRa1} and \cite[Proposition 3.7]{DieningGmeineder} for Poincar\'e-type inequalities on $\L^p$-spaces for star-shaped or John domains respectively. We proceed and derive the announced Korn-type inequality in Orlicz spaces. 
\begin{thm}[Korn-type inequality]
	Let $r>0$ as well as $\Psi$ and $\Phi$ be Young functions such that \eqref{eq:CianchiAssumptionYoungFunctions} is satisfied. Then for each $\beta>0$ there exists a constant $c=c(\Phi,\Psi,n,r,\Aop)>0$ such that 
	\begin{align}\label{eq:KornType2_Orlicz}
		\inf_{\rfrak\in\ker(\Aop)} \|\D (u-\rfrak)\|_{\L^\Phi(\ball_r(x_0);\Lin(\R^n;V))}\leq c \bigg(1+\frac{1}{r}\bigg)\|\Aop u \|_{\L^\Psi(\ball_{2r}(x_0);W)}
	\end{align}
	for all $u\in\sobo^{\Aop,\Psi}(\ball_{2r}(x_0))$, where $\Aop$ is a  $\C$-elliptic differential operator of the form \eqref{eq:DifferentialOperator}.
\end{thm}
\begin{proof}
	Let $\varrho\in\hold^\infty_c(\R^n;[0,1])$ be a cut-off function satisfying $\chi_{\ball_{r}(x_0)}\leq \varrho\leq \chi_{\ball_{2r}(x_0)}$ and $\abs{\nabla\varrho}\leq\frac{2}{r}$. Moreover, we recall from \eqref{eq:Poincare_Projection} that the  Poincar\'e-type inequality 
	\begin{equation}\label{eq:Poincare_KornProof}
		\|u-\Pi_\Aop^\ball u\|_{\L^\Psi(\ball_{2r}(x_0);V)} \leq c\,\omega_n\, r\|\Aop u\|_{\L^\Psi(\ball_{2r}(x_0);W)},
	\end{equation}
	holds true. Therefore, taking into account \eqref{eq:Korn_Type_Orlicz_Paul_Norm} on $\Omega=\ball_{5r}(x_0)$ applied for the function $\varrho(u-\Proj_\ball u)\in\sobo^{1,\Psi}_0(\ball_{5r}(x_0);V)$ in conjunction with \eqref{eq:Poincare_KornProof}, leads to 
	\begin{align*}
		\|\D (u-\Pi_\Aop^\ball u)\|_{\L^\Phi(\ball_{r}(x_0))} 
		&\leq \|\D (\varrho(u-\Pi_\Aop^\ball u))\|_{\L^\Phi(\ball_{5r}(x_0)}\\
		& \leq c \,\|\Aop (\varrho(u-\Pi_\Aop^\ball u))\|_{\L^\Psi(\ball_{5r}(x_0);W)}\\
		&\leq c \,\bigg( \frac1r\left\|u-\Pi_\Aop^\ball u\right\|_{\L^\Psi(\ball_{2r}(x_0);V)} +  \|\Aop u\|_{\L^\Psi(\ball_{2r}(x_0); W)}\bigg)\\
		&\leq c\,\bigg(1+\frac{1}{r}\bigg) \|\Aop u\|_{\L^\Psi(\ball_{2r}(x_0);W)}.
	\end{align*}
	This yields the desired inequality \eqref{eq:KornType2_Orlicz}, as we can always bound the left-hand side by the infimum taken over all $\rfrak\in\ker(\Aop)$, which finishes the proof. 
\end{proof}

\pagebreak \noindent As an immediate consequence, cf.~also \cite [Example 3.11]{Cianchi}, we observe 
\begin{cor}\label{cor:KornType_Orlicz_C_Elliptic}
	Let $r>0$, $x_0\in\R^n$ and $\ball_r(x_0)\subset\R^n$ be a ball. Then for all $\beta>0$ there exists a constant $c=c(\beta,r,n)>0$ such that
	\begin{align}\label{eq:exponentialKorn_type}
		\inf_{\rfrak\in\ker(\Aop)} \|\D (v-\rfrak)\|_{\exp\L^{\frac{\beta}{\beta+1}}(\ball_r(x_0);\Lin(\R^n;V))}
		\leq c\,\bigg(1+\frac{1}{r}\bigg)\|\Aop v \|_{\exp\L^{\beta}(\ball_{2r}(x_0);W)}
	\end{align}
	for all $v\in\exp\L^\beta(\ball_{2r}(x_0);V)$ with $\Aop \in \exp\L^\beta(\ball_{2r}(x_0);W$) 
	, where $\Aop$ is a $\C$-elliptic operator of the form \eqref{eq:DifferentialOperator}.
\end{cor}
\begin{rem}[Korn-type inequality on Lipschitz domains]
	It is possible to establish  a version of \eqref{eq:KornType2_Orlicz} without increasing the domains of integration on the right-hand side.
	In order to do so, one can follow the lines of the proof of \cite[Thm. 3.3]{Cianchi}.
	The crucial idea to establish Korn-type inequalities is Sobolev's integral representation formula cf. \cite[Sec. 1.1.10]{Mazya}, namely for all $u\in\hold^\infty_c(\R^2;\R^2)$ it holds
	\begin{align}\label{eq:representation}
		u=\frac{1}{\pi}\sum_{\abs{\alpha}=2}\frac{1}{\alpha!}\frac{x^\alpha}{\abs{x}^2}\ast \partial^\alpha u.
	\end{align}
	If the projection $\A:\R^{2\times 2}\to\R^{2\times2}$ induce by $\Aop\coloneqq\A[\cdot\otimes\nabla]$ a $\C$-elliptic differential operator on $\R^2$, then by our Lemma \ref{lem:Cellipt2d} we can express $\partial^\alpha u$ as linear combination 
	\begin{align*}
		\partial^\alpha u = \partial_1 \ell_1(\A[\D u]) + \partial_2 \ell_2(\A[\D u]), \qquad \abs{\alpha}=2,
	\end{align*}
	with suitable linear maps $\ell_j:\R^{2\times 2} \to\R^2$ for $j=1,2$. Thus, reinserting in \eqref{eq:representation} and integrating by parts we obtain
	\begin{align*}
		u=-\frac{1}{\pi}\sum_{\abs{\alpha}=2}\frac{1}{\alpha!}\sum_{j=1}^2\partial_j\frac{x^\alpha}{\abs{x}^2}\ast \ell_j(\A[\D u]),
	\end{align*}
to arrive at a similar statement as in \cite[Eq. (6.6)]{Cianchi}.
Furthermore, since our induced differential operator $\Aop$ is $\C$-elliptic, its kernel is finite dimensional, cf. also Corollary \ref{cor:kernel}, so that we can also deduce Korn-type inequalities of the second type in Orlicz spaces. 
First, via an Orlicz version of the Poincaré's inequality involving the $\C$-elliptic differential operator $\Aop$, we have
	\begin{equation}\label{eq:Poincare2.0}
		\inf_{v\in\mathscr{K}}\norm{u-v}_{\exp\lebe^\beta(\Omega;\R^2)}\le c(\beta,\Aop,\Omega)\norm{\Aop u}_{\exp\lebe^\beta(\Omega;\R^{2\times 2})}
	\end{equation}
	for every $u\in \lebe^{\Phi_\beta}(\Omega;\R^2)$ with $\Aop u\in\lebe^{\Phi_\beta}(\Omega;\R^{2\times2})$ where $\mathscr K$ denotes the finite dimensional nullspace of $\Aop$
	\begin{equation*}
		\mathscr K \coloneqq \{ v(x)=Bx+b, B\in\ker \A, b \in\R^2 \},
	\end{equation*}
 cf.~\Cref{cor:kernel}. 
	Let us shortly comment how we can again follow the lines of the proof of \cite[Thm 3.3]{Cianchi} emphasising the needed adjustments: The first ingredient is Smith's representation formula on cones for $\C$-elliptic differential operators \cite[Thm I]{Smith}. The next crucial step is to construct an extension operator like in \cite[Ch II Rem 1.3, Thm 1.2, Thm 2.2]{Temam} which relies on a limiting Sobolev type inequality \cite[Eq. (1.32)]{Temam}. The latter result has been generalised by \textsc{Van Schaftingen} \cite[Thm 1.3]{VS} and reads in our situation
	\begin{equation}\label{eq:limitigSobo}
		\norm{u}_{\lebe^{2}(\R^2;\R^2)}\le c(\Aop) \norm{\Aop u}_{\lebe^1(\R^2;\R^{2\times2})} \qquad \forall \ u\in\hold^\infty_c(\R^2;\R^2)
	\end{equation}
	and holds true if and only if $\Aop$ is an elliptic and cancelling operator. In particular \eqref{eq:limitigSobo} is fulfilled for $\C$-elliptic differential operators, since $\C$-ellipticity implies both ellipticity and cancellation, cf. \cite{GRV}. Hence, we can apply the arguments from the proof of \cite[Thm 3.3]{Cianchi} to derive similar to \cite[Eq. (6.15)]{Cianchi} the estimate
	\begin{equation}\label{eq:Cianchi(6.15)}
		\norm{\D u}_{\exp\lebe^{\frac{\beta}{\beta+1}}(\Omega;\R^{2\times 2})}\le c(\beta,\Aop,\Omega)\,\left(\norm{u}_{\exp\lebe^\beta(\Omega;\R^{2\times 2})}+\norm{\Aop u}_{\exp\lebe^\beta(\Omega;\R^{2\times 2})}\right)
	\end{equation}
	for every $u\in \exp\L^\beta(\Omega;\R^2)$ with $\Aop u\in\exp\L^\beta(\Omega;\R^{2\times2})$. Thus, applying \eqref{eq:Cianchi(6.15)} to $u-v$ for a arbitrary $v$ from the nullspace of $\Aop$ we conclude in view of the Poincaré inequality \eqref{eq:Poincare2.0}
	\begin{equation}\label{eq:Korn2}
		\inf_{v\in\mathscr{K}}\norm{\D(u-v)}_{\exp\lebe^{\frac{\beta}{\beta+1}}(\Omega;\R^{2\times 2})}\le c(\beta,\Aop,\Omega)\norm{\Aop u}_{\exp\lebe^\beta(\Omega;\R^{2\times 2})},
	\end{equation}
	which is the desired Korn's inequality of the second type in Orlicz spaces.
\end{rem}
\subsection{Proof of the main \Cref{thm:MainTheorem}.} 
\begin{proof}
For the sake of readability we divide the proof into three parts. To this end, we consider $x_0\in\Omega$ and $r>0$ such that $\ball_{2r}(x_0)\Subset \Omega$. 

	\textit{Step 1: Preliminary estimate.}
	As a first step we let $(v_j)_{j\in\N}$ be the Ekeland approximation sequence from \Cref{prop:EkelandSequence} and consider the convex auxiliary function 
	\[
	V_\mu(\xi) \coloneqq (1+\abs{\xi}^2)^{\frac{2-\mu}{4}}\quad\mbox{for} \quad\xi\in\R^{2\times 2}_\A.
	\]
	We recall from \Cref{lem:NonUniformSecEstimates} that $v_j\in \sobo^{2,2}_\loc(\Omega;\R^2)$ and therefore, we can estimate for $k\in\{1,2\}$ the derivatives by 
	\begin{align}\label{eq:BoundDerivativeOfV}
		\begin{split}
			\abs{\del_k V_\mu( \A[\D v_j])}^2 
			&\leq \left(\frac{2-\mu}{2}\right)^2 \big(1+\abs{\A[\D v_j]}^2\big)^{\frac{-2-\mu}{2}} \abs{\A[\D v_j]}^2 \abs{\del_k \A[\D v_j]}^2 \\
			&\leq c(\mu) \frac{\abs{\del_k \A[\D v_j]}^2}{(1+\abs{\A[\D v_j]}^2)^{\frac{\mu}{2}}}.
		\end{split}
	\end{align}
	Furthermore, we observe that \eqref{eq:AreaStrictConvergence}, \eqref{eq:minimisingSequenceWithSmoothBoundary} and \eqref{eq:Ekeland_1} imply  $v_j\to u$ in $\sobo^{-2,1}(\Omega;\R^2)$. Hence, using \eqref{eq:EkelandSequence_Property_1} together with the compact embedding $\sobo^{\Aop,1}(\Omega)\hookrightarrow\lebe^1(\Omega;\R^2)$, cf.~\cite[Theorem 1.1]{GmRa1}, we can extract a non-relabeled subsequence of $(v_j)_{j\in\N}$ such that $v_j\weakstar u$ in  $\BV^\Aop(\Omega)$ as $j\to\infty$. Using \eqref{eq:BoundDerivativeOfV} with \eqref{eq:WeightedSecondOrder_MuElliptic} for $\mu\in(1,2)$, we obtain
	\begin{align}\label{eq:Proof_MainTheorem1}
	\begin{split}
	&\|V_\mu(\A[\D v_j])\|^2_{\sobo^{1,2}(\ball_{r}(x_0))} \\
	& \hspace{1cm}= \|V_\mu(\A[\D v_j])\|^2_{\lebe^2(\ball_{r}(x_0))} + \|\nabla V_\mu(\A[\D v_j])\|^2_{\lebe^2(\ball_{r}(x_0))} \\
	&\hspace{1cm}\leq \int_{\ball_{r}(x_0)} (1+\abs{\A[\D v_j]}^2)^{\frac{1}{2}} \dx{x} + c(\mu) \int_{\ball_{2r}(x_0)}\frac{\abs{\D \A[\D v_j]}^2}{(1+\abs{\A[\D v_j]}^2)^{\frac{\mu}{2}}}\dx{x} \\
	&\hspace{1cm}\leq c \bigg(1+\frac{1}{r^2}\bigg) \int_{\ball_{2r}(x_0)} \abs{\A[\D v_j]}\dx{x} + \frac{c}{r^2} \bigg[\frac{2}{j^2} + \bigg(1+\frac{r^2}{j}+\frac{r^3}{j}\bigg)\int_{\ball_{2r}(x_0)} \frac{\abs{v_j}}{r}\dx{x}\bigg], 
	\end{split}
	\end{align}
	with a constant $c=c(\Lambda,\A,\mu)>0$. Therefore, the sequence $(V_\mu(\A[\D v_j]))_{j\in\N}$ is uniformly bounded in $\sobo^{1,2}(\ball_r(x_0))$.
	Now, the Trudinger theorem in two dimensions, cf.~\cite{Trudinger}, gives the embedding $\sobo^{1,2}(\ball_r(x_0))\hookrightarrow\exp \lebe^2(\ball_{r}(x_0)$ and hence
	\begin{align}\label{eq:ExponentialBoundV_Function_C_Elliptic}
	 \|V_\mu(\A[\D v_j])\|_{\exp\lebe^2(\ball_{r}(x_0))}
	 &\leq
	 c(r) \|V_\mu(\A[\D v_j])\|_{\sobo^{1,2}(\ball_r(x_0))}.
	\end{align}
	As a next step, we utilise $(V_\mu(\xi))^2\geq \abs{\xi}^{2-\mu}$ to conclude the estimate 
	\begin{align}\label{eq:OrliczEstimate_C_Elliptic}
		\| \A[\D v_j]\|_{\exp\lebe^{2-\mu}(\ball_{r}(x_0);\R^{2\times 2}_\A)}\leq \|V_\mu(\A[\D v_j])\|_{\exp \lebe^2(\ball_{r}(x_0))}^\frac{2}{2-\mu},
	\end{align}
	which can be seen as follows: 
	Setting $\lambda_j\coloneqq \|V_\mu(\A[D v_j])\|^{\frac{2}{2-\mu}}_{\exp\lebe^2(\ball_r(x_0))}$ and using the definition of the Luxemburg norm, yields the chain of inequalities
	\begin{align*}
		1\geq \int_{\ball_r(x_0)} \exp\left(\frac{(V_\mu(\A[\D v_j]))^2}{\|V_\mu(\A[\D v_j])\|^2_{\exp\lebe^2(\ball_r(x_0))}}\right)\dx{x}\geq \int_{\ball_r(x_0)} \exp\left(\frac{\abs{\A [\D v_j]}^{2-\mu}}{\lambda_j^{2-\mu}}\right)\dx{x}, 
	\end{align*}
implying $\|\A[\D  v_j]\|_{\exp\lebe^{2-\mu}(\ball_{r}(x_0))}\leq \lambda_j$ and hence, the claimed estimate \eqref{eq:OrliczEstimate_C_Elliptic} is proved. Recalling $v_j\weakstar u$ in $\BV^\Aop(\Omega)$ toegether with \eqref{eq:ExponentialBoundV_Function_C_Elliptic} and \eqref{eq:OrliczEstimate_C_Elliptic}, allows to pass to the limit in \eqref{eq:Proof_MainTheorem1} which yields
\begin{align}\label{eq:Liminf_MainTheorem}
	\begin{split}
	&\liminf_{j\to\infty} \|\A[\D v_j]\|_{\exp\L^{2-\mu}(\ball_r(x_0);\R^{2\times 2}_\A)} \\
	&\hspace{1cm}\leq c(\Lambda,\A,\mu) \bigg[\bigg(1+\frac{1}{r^2}\bigg) \abs{\Aop u}(\overline{\ball}_{2r}(x_0))+ \dashint_{\ball_{2r}(x_0)} \frac{\abs{u}}{r}\dx{x} \bigg]^{\frac{1}{2-\mu}}\eqqcolon \mathfrak{m}<\infty. 
	\end{split}
\end{align}


Since $\Aop$ is a linear differential operator and $\mathfrak{m}>0$ we can replace $v_j$ by $\tilde{v}_j=\frac{v_j}{\mathfrak{m}}$ and consequently, we have  $\|\Aop \tilde{v}_j\|_{\exp\lebe^{2-\mu}(\ball_{r}(x_0))}\leq 1$.  Applying \cite[Lemma 8.8]{BennetSharpley} leads then to 
\begin{align}\label{eq:IntegralBound_C_Elliptic}
	\begin{split}
	\int_{\ball_r(x_0)} \exp(\abs{\A [\D\tilde{v}_j]}^{2-\mu})\dx{x} &= \int_{\ball_r(x_0)} \exp\left(\frac{\abs{\A [\D v_j]}^{2-\mu}}{\mathfrak{m}^{2-\mu}}\right)\dx{x}\\
	&\leq \|\A[\D \tilde{v}_j]\|_{\exp\lebe^{2-\mu}(\ball_{r}(x_0);\R^{2\times 2}_\A)}. 
	\end{split}
\end{align}

	\textit{Step 2: Exponential integrability of $\Aop u$.}
	We aim to use the Reshetnyak lower semi-continuity \Cref{thm:Reshetnyak} to conclude that $\Aop^s u\equiv 0$ in $\ball_{r}(x_0)$. To this end, we introduce the function 
	\begin{equation*}
		\Phi\colon\R^{2\times 2}_\A\to\R, \quad \Phi(z) = \exp\left(\frac{\abs{z}^{2-\mu}}{\mathfrak{m}^{2-\mu}}\right),
	\end{equation*}
	and observe that the corresponding recession function is given by
	\begin{align*}
		\Phi^\infty(z) \coloneqq \lim_{t\to\infty} \frac{\Phi(tz)}{t} 
		= \lim_{t\to\infty} \frac1t\exp\left(\frac{(t\abs{z})^{2-\mu}}{\mathfrak{m}^{2-\mu}}\right)
		=
		\begin{cases}
			0\quad&\mbox{if}\quad \abs{z}=0\\
			+\infty\quad&\mbox{if}\quad\abs{z}>0.
		\end{cases}
	\end{align*}
	Applying Reshetnyak's \Cref{thm:Reshetnyak} in combination with \eqref{eq:IntegralBound_C_Elliptic} we observe
	\begin{align*}
		\int_{\ball_{r}(x_0)} \Phi(\A [\D u])\dx{x} + \int_{\ball_r(x_0)} \Phi^\infty\left(\frac{\dx{\Aop^s u}}{\dx{\abs{\Aop^s u}}}\right) \dx{\abs{\Aop^s u}} &= \Phi(\Aop u)(\ball_r(x_0))\\
		&\leq \liminf_{j\to\infty} \Phi(\Aop v_j)(\ball_r(x_0))\leq 1,
	\end{align*}
	where for the last inequality we have used \eqref{eq:Liminf_MainTheorem} and \eqref{eq:IntegralBound_C_Elliptic}. 
	Hence, by the definition of the recession function we must have $\Aop^s u\equiv 0$ on $\ball_r(x_0)$ and 
	\begin{align*}
		\int_{\ball_r(x_0)}\Phi(\A[\D u] )\dx{x}=\int_{\ball_r(x_0)} \exp\left(\frac{\abs{\A[\D u]}^{2-\mu}}{\mathfrak{m}^{2-\mu}}\right)\dx{x}\leq 1.  
	\end{align*}
	Recalling once more the definition of the Luxemburg norm, we infer $\A[\D u] \in\exp\lebe^{2-\mu}(\ball_r(x_0);\R^{2\times 2}_\A)$ with estimate $\|\A[\D u]\|_{\exp\lebe^{2-\mu}(\ball_r(x_0);\R^{2\times 2}_\A)}\leq \mathfrak{m}$, i.e
	\begin{equation}
		\|\A[\D v_j]\|_{\exp\L^{2-\mu}(\ball_r(x_0);\R^{2\times 2}_{\A})} \leq c \bigg[\bigg(1+\frac{1}{r^2}\bigg) \abs{\Aop u}(\overline{\ball}_{2r}(x_0))+ \dashint_{\ball_{2r}(x_0)}\frac{\abs{u}}{r}\dx{x}\bigg]^{\frac{1}{2-\mu}},
	\end{equation}
	for a constant $c=c(\Lambda,\A,\mu)>0$.

	\textit{Step 3: Exponential integrability of \,$\D u$.} Let us remember that $\mathscr{K}$ is finite dimensional, cf.~\Cref{cor:kernel}, and therefore, all norms are equivalent. In particular, there holds
	\begin{equation*}
		\|v\|_{\exp\lebe^{\frac{2-\mu}{3-\mu}}(\ball_r(x_0);\R^2)} + r\,\|\D v\|_{\exp\lebe^{\frac{2-\mu}{3-\mu}}(\ball_r(x_0);\R^{2\times 2}_{\A})}\leq c(\mu)~~\dashint_{\ball_{r}(x_0)}\abs{v}\dx{x}
	\end{equation*}
	for all $v\in\mathscr{K}$. Moreover, there exists a bounded linear projection operator $\pi\colon \lebe^1(\ball_r(x_0);\R^2)\to \mathscr{K}$ with $v\mapsto \pi_v$ such that 
	\begin{equation*}
		\dashint_{\ball_{r}(x_0)} \abs{\pi_v}\dx{x} \leq c(r)~\dashint_{\ball_{r}(x_0)}\abs{v}\dx{x},
	\end{equation*}
	for all $v\in\lebe^1(\ball_{r}(x_0);\R^2)$, cf.~\cite[Appendix 9.1]{Gmeineder}. 
	To obtain the precise statement of \eqref{eq:FinalEstimateMainTheorem} we use the Korn-type inequality from \eqref{eq:Korn2} with $\beta\coloneqq 2-\mu>0$ and noticing $\tfrac{\beta}{\beta+1}= \tfrac{2-\mu}{3-\mu}$ gives
	\begin{align*}
		&\|\D u\|_{\exp\lebe^{\frac{2-\mu}{3-\mu}}(\ball_{r}(x_0);\R^{2\times 2})}\\
		&\hspace{1cm}\leq \|\D(u-\pi_u)\|_{\exp\lebe^{\frac{2-\mu}{3-\mu}}(\ball_{r}(x_0);\R^{2\times 2})} 
		+
		\|\D \pi_u\|_{\exp\lebe^{\frac{2-\mu}{3-\mu}}(\ball_{r}(x_0);\R^{2\times 2})}\notag\\
		&\hspace{1cm}\leq c\bigg(\left(1+\frac1r\right)\|\Aop u\|_{\exp\lebe^{2-\mu}(\ball_{2r}(x_0);\R^{2\times 2}_\A)} +  \dashint_{\ball_{r}(x_0)}\frac{\abs{u}}{r}\dx{x} \bigg)\\
		&\hspace{1cm}\leq c \left(\bigg(1+\frac{1}{r^2}\bigg)\bigg[\bigg(1+\frac{1}{r^2}\bigg) \abs{\Aop u}(\overline{\ball}_{2r}(x_0)) + \dashint_{\ball_{2r}(x_0)} \frac{\abs{u}}{r} \dx{x}\bigg]^{\frac{1}{2-\mu}} + \dashint_{\ball_{r}(x_0)} \frac{\abs{u}}{r}\dx{x}\right),
	\end{align*}
	with $c = c(\Lambda,\mu,\Aop)>0$.  This finishes the proof.

\end{proof}
\begin{rem}
 In this short note, we have investigated into the regularity of $\BV^\Aop$-minimisers in two dimensions. Here, the crucial idea was to rely on the one-dimensional (almost) complementary part of the part maps that induce $\C$-elliptic, first-order differential operators. Furthermore, we have shown that this strategy is also applicable for the deviatoric gradient in all dimensions. The case of the symmetric gradient in all dimensions was treated in \cite{BEG}. However, the case of a general $\C$-elliptic operator in any dimensions deserves a new approach. In particular, it remains an open problem to treat the trace-free symmetric gradient in three dimensions.
\end{rem}

\section*{Acknowledgments}
The authors are very grateful to Lisa Beck and Franz Gmeineder for inspiring discussions and valuable comments. The first author is thankful for financial support by the Karlsruhe Institute of Technology for a visit in November 2024. The results and discussions presented in this paper are part of the first author’s PhD thesis written under the supervision of Lisa Beck and Franz Gmeineder. The initial idea for this joint project originated at the \textit{Trends in Analysis} conference organised by the second author in 2023 at the University of Duisburg-Essen.

\end{document}